\documentclass[twoside]{article}

\usepackage{amsfonts}
\usepackage{amsthm}
\usepackage{amsmath}
\usepackage{amssymb}

\title{\bf On the existence of extreme waves and the Stokes conjecture with vorticity }
\author{Eugen Varvaruca}\date{}

\newtheorem{theorem}{Theorem}[section]
\newtheorem{lemma}[theorem]{Lemma}
\newtheorem{proposition}[theorem]{Proposition}
\newtheorem{remark}[theorem]{Remark}

\numberwithin{equation}{section}

\newcommand{\hg}{\hat\Gamma}

\newcommand{\Om}{\Omega}

\newcommand{\om}{\omega}
\newcommand{\la}{\lambda}
\newcommand{\te}{\theta}
\newcommand{\vfi}{\varphi}
\newcommand{\be}{\begin{equation}}
\newcommand{\ee}{\end{equation}}

\newcommand{\ds}{\displaystyle}

\newcommand{\mcs}{{\mathcal S}}
\newcommand{\mcu}{{\mathcal U}}

\newcommand{\mcb}{{\mathcal B}}
\newcommand{\mcd}{{\mathcal D}}
\newcommand{\mcg}{{\mathcal G}}
\newcommand{\mcl}{{\mathcal L}}
\newcommand{\mcr}{\mathcal{R}}

\newcommand{\mch}{{\mathcal H}}
\newcommand{\mck}{{\mathcal K}}
\newcommand{\Lip}{{\rm Lip}}
\newcommand{\indj}{_{j\geq 1}}
\newcommand{\loc}{{\textnormal{loc}}}
\newcommand{\tips}{\tilde\psi}
\newcommand{\wido}{\widetilde\Om}
\newcommand{\wids}{\widetilde\mcs}
\newcommand{\ovo}{\overline\Om}
\newcommand{\ha}{H_{\mathbb{C}}}
\newcommand{\h}{h_{\mathbb{C}}}

\newcommand{\psix}{\psi_{_{X}}}
\newcommand{\psiy}{\psi_{_{Y}}}

\newcommand{\bdc}{\mathbb{C}}
\newcommand{\bdr}{\mathbb{R}}

\newcommand{\mcc}{\mathcal{C}}

\newcommand{\veps}{\varepsilon}

\newcommand{\indn}{_{n\geq 1}}

\newcommand{\bese}{\begin{subequations}}
\newcommand{\ese}{\end{subequations}}
\newcommand{\non}{\nonumber}

\renewcommand{\Im}{\textnormal{Im}\,}

\begin{document}
\maketitle \centerline{Department of Mathematical Sciences,
University of Bath} \centerline{Claverton Down, Bath BA2 7AY, United
Kingdom}\centerline{Email address: {\tt mapev@maths.bath.ac.uk}}
\centerline{}
\begin{abstract} This is a study of singular solutions of the problem of traveling
gravity water waves on flows with vorticity. We show that, for a
certain class of vorticity functions, a sequence of regular waves
converges to an extreme wave with stagnation points at its crests.
We also show that, for any vorticity function, the profile of
 an extreme wave must have
either a corner of $120^\circ$ or a horizontal tangent at any
stagnation point about which it is supposed symmetric. Moreover, the
profile necessarily has a corner of $120^\circ$ if the vorticity is
nonnegative near the free surface.
\end{abstract}

\section{Introduction}

This article addresses the classical hydrodynamical problem
concerning traveling two-dimensional gravity water waves with
vorticity. There has been considerable interest on this problem in
recent years, starting with the systematic study of Constantin and
Strauss \cite{CoS}.

When the water depth is finite, which is the setting of \cite{CoS},
the problem arises from the following physical situation. A wave of
permanent form moves with constant speed on the surface of an
incompressible, inviscid, heavy fluid, the bottom of the fluid
domain being horizontal. With respect to a frame of reference moving
with the speed of the wave, the flow is steady and occupies a fixed
region $\Om$ in $(X,Y)$-plane, which lies above a horizontal line
$\mcb_F:=\{(X,F):X\in\bdr\}$, where $F$ is a constant, and below
some a priori unknown free surface
$\mcs:=\{(u(s),v(s)):s\in\mathbb{R}\}$.
 Since the fluid is incompressible, the
flow can be described by a \emph{stream function} $\psi$
 which satisfies the following equations and boundary conditions:
 \begin{subequations}\label{apb}
\begin{align}
& \Delta\psi=-\gamma(\psi)\quad\text{in
}\Om,\label{ap0}\\&0\leq\psi\leq B\quad\text{in }\Om,\label{ap*}\\ &
\psi=B\quad \text{on }\mcb_F,\label{ap1}\\&\psi=0\quad \text{on
}\mcs, \label{ap2}\\&\vert\nabla\psi\vert^{2}+2gY=Q \quad\text{on }
\mcs\label{ap3},
 \end{align}
\end{subequations}
where $Q$ is a constant, $B$, $g$ are positive constants and
$\gamma:[0,B]\to\bdr$ is a function. The meaning of equation
(\ref{ap0}) is that the \emph{vorticity} of the flow
$\om:=-\Delta\psi$ and the stream function $\psi$ are functionally
dependent. It is customary \cite{CoS} to assume that the constants
$g, B$ and the function $\gamma$, called a \emph{vorticity
function}, are given. The problem consists in determining the curves
$\mcs$ for which there exists a function $\psi$ in $\Om$ satisfying
(\ref{apb}) for some values of the parameters $Q$ and $F$. Any such
solution quadruple $(\mcs, \mcb_F, \psi, Q)$ of (\ref{apb}) gives
rise to a traveling-wave solution of the two-dimensional Euler
equations for a heavy fluid with a free surface, see \cite{CoS} for
details. In particular, the relative velocity of the fluid particles
is given by $(\psiy,-\psix)$.
 Among
various types of waves, of main interest are the periodic waves, for
which $\mcs$ is periodic in the horizontal direction, and the
solitary waves, for which $\mcs$ is asymptotic to a horizontal line
at infinity.

In the related problem of waves of infinite depth, one seeks a curve
$\mcs$ such that in the domain $\Om$ below $\mcs$ there exists a
function $\psi$ which satisfies (\ref{ap0}), (\ref{ap2}),
(\ref{ap3}) and \begin{align}&\psi\geq 0\quad\text{in
}\Om,\tag{1.1b'}\\&\nabla\psi(X,Y)\to (0,-C)\quad\text{as $Y\to
-\infty$, uniformly in $X$,} \tag{1.1c'}\end{align} where
$\gamma:[0,\infty)\to\bdr$ is a given function and $C$ is a
parameter. Of main interest are the periodic waves.

When $\gamma\equiv 0$, the corresponding flow is called
\emph{irrotational}. Nowadays the mathematical theory dealing with
this situation contains a wealth of results, mostly obtained during
the last three decades. The first existence result for waves of
large amplitude was given by Krasovskii \cite{Kr}. Then, global
bifurcation theories for regular waves of various types were given
by Keady and Norbury \cite{KN} and by Amick and Toland
\cite{AT1,AT2}. Moreover, it was shown by Toland \cite{T78} and by
McLeod \cite{ML} that in the closure of these continua of solutions
there exist waves with \emph{stagnation points} (i.e., points at
which the relative fluid velocity is zero) at their crests. The
existence of such waves, called \emph{extreme waves}, was predicted
by Stokes \cite{S}, who also conjectured that their profiles
necessarily have corners with included angle of $120^\circ$ at the
crests. This conjecture was proved independently by Amick, Fraenkel,
and Toland \cite{AFT}, and by Plotnikov \cite{P82}. In more recent
developments,  the method of \cite{AFT} was simplified and
generalized in \cite{EV}, while Fraenkel \cite{Fr} gave a direct
proof of the existence of an extreme wave (of infinite depth), with
corners of $120^\circ$ at the crests, without relying on existence
results for regular waves.

 When $\gamma\not\equiv 0$, the flow is called \emph{rotational} or \emph{with vorticity},
 and advances in the
mathematical theory have been made only in the last few years. The
existence of global continua of solutions was proved by Constantin
and Strauss \cite{CoS} for the periodic finite depth problem, and by
Hur \cite{H} for the periodic infinite depth problem. The wave
profiles in \cite{CoS,H} have one crest and one trough per minimal
period, are monotone between crests and troughs and symmetric with
respect to vertical lines passing through any crest.
 The continuum of solutions in \cite{CoS} contains waves for
  which
the values of $\max_{\overline\Om}\psiy$ are arbitrarily close to
$0$
 and, at least in certain situations \cite{EV4}, the
values of $|\nabla\psi|$ at the crests are also arbitrarily close to
$0$. Thus it is natural to expect that, as in the irrotational case,
waves with stagnation points at their crests, referred to as
\emph{extreme waves}, exist for many vorticity functions, and that
they can be obtained as limits, in a suitable sense, of certain
sequences of regular waves found in \cite{CoS}. In the case of
constant vorticity, numerical evidence \cite{KS, SS, TSP, VB1, VB2,
VB3} strongly points to the existence of extreme waves for any
negative vorticity and for small positive vorticity, and also
indicates that, for large positive vorticity, continua of solutions
bifurcating from a line of trivial solutions develop into
overhanging profiles (a situation which is not possible in the
irrotational case, see \cite{EV3} for references) and do not
approach extreme waves. The above mentioned numerical computations
support the formal speculation in various places in the fluid
mechanics literature \cite{Del}, \cite[\S 14.50]{MT} that extreme
waves with vorticity must also have corners with angles of
$120^\circ$ at the crests. We refer to this claim as the
\emph{Stokes conjecture}, although Stokes himself seems to have made
it explicitly only for irrotational waves.

This article is, to the best of our knowledge, the first rigorous
study of the existence of extreme waves with vorticity and their
properties. Attention is restricted here to the case of periodic
waves in water of finite depth, though it is clear that similar
arguments can be used in related situations, such as solitary waves
of finite depth or periodic waves of infinite depth.

A fundamental difficulty when trying to extend to the general case
of waves with vorticity known results for irrotational waves is that
new methods are needed. Indeed, the irrotational case is the only
one in which conformal mappings can be used to equivalently
reformulate the free-boundary problem as an integral equation
\cite{KN,AT2}, originally due to Nekrasov \cite{Ne}, for a function
which gives the angle between the tangent to the free boundary and
the horizontal. The existence of large-amplitude regular waves, the
existence of extreme waves and the Stokes conjecture are then proved
by using hard analytic estimates for this integral equation
\cite{T}. For waves with vorticity, the existence of large-amplitude
regular waves \cite{CoS} is based on a study of another equivalent
reformulation of the problem, originally due to Dubreil-Jacotin
\cite{DJ}, as a quasilinear second order elliptic partial
differential equation with nonlinear boundary conditions in a fixed
domain. However, this reformulation of the problem does not seem
suitable to describe extreme waves.

Our first task, pursued in Section 2, is thus to identify
generalized formulations of problem (\ref{apb}), under minimal
regularity assumptions, which are suitable for the description of
extreme waves. We introduce two types of solutions, called
respectively \emph{Hardy-space solutions} and \emph{weak solutions}.
An extensive theory of Hardy-space solutions has been given in the
case of irrotational waves by Shargorodsky and Toland \cite{ST}, and
further developed in \cite{EV,EV2,EV3}. The notion of a weak
solution of (\ref{apb}) is inspired by the article of Alt and
Caffarelli \cite{AC}, who considered a class of free boundary
problems in bounded domains (in any number of dimensions) for
harmonic functions satisfying simultaneously on a free boundary a
Dirichlet boundary condition of type (\ref{ap2}) and a boundary
condition of a more general type than (\ref{ap3}). Each of these
solution types has certain advantages over the other, and the main
result of Section 2 is that the two coincide. The material in this
section pervades the rest of the article.

In Section 3 we prove, by means of the maximum principle, an a
priori estimate concerning the \emph{pressure} in the fluid. This
result, which extends to the general case some very recent results
in \cite{EV4} for vorticity functions which do not change sign,
plays a pivotal role in the investigation of the existence of
extreme waves and the Stokes conjecture.

In Section 4 we study the existence of extreme waves. We consider a
sequence of solutions $\{(\mcs_j,\mcb_0, \psi^j, Q_j)\}_{j\geq 1}$
of (\ref{apb}), which have similar properties to the solutions in
the continuum in \cite{CoS}. In particular, for all $j\geq 1$,
$\mcs_j=\{(X,\eta_j(X)):X\in\bdr\}$, where
\[ \text{$\eta_j\in C^1(\bdr)$ is $2L$-periodic, even and } \eta_j'<
0\text{ on }(0,L).\] In Theorem \ref{exi} we prove, under the
\emph{assumption} that \be\text{$\{Q_j\}\indj$ is bounded
above,}\label{fhe}\ee that a subsequence of
$\{(\mcs_j,\mcb_0,\psi^j, Q_j)\}_{j\geq 1}$ necessarily converges in
a specified sense
 to a weak solution
 $(\wids, \mcb_0, \tips, \widetilde Q)$ of (\ref{apb}). Moreover, the additional
 \emph{assumption} that
\be\text{$|\nabla\psi^j(0,\eta_j(0))|\to 0$ as
$j\to\infty$},\label{fhc}\ee ensures that $(\widetilde\mcs,\mcb_0,
\tips, \widetilde Q)$ is an extreme wave. This result is far from
trivial. The most difficult steps in the proof are the definition of
$\wids$ as non-self-intersecting curve in the absence of any uniform
bound on the slopes of $\{\mcs_j\}\indj$, and the recovery of the
free-boundary condition (\ref{ap3}) in a weak sense along $\wids$.

Combining Theorem \ref{exi} with existing results in the literature
\cite{CoS,EV4} on the validity of (\ref{fhe}) and (\ref{fhc}) for a
sequence in the continuum in \cite{CoS}, we obtain in Theorem
\ref{zxc} the existence of extreme waves arising as limits of
regular waves in the case when $\gamma(0)<0$, $\gamma(r)\leq 0$ and
$\gamma'(r)\geq 0$ for all $r\in [0,B]$. However, these assumptions
on $\gamma$ also ensure the existence of \emph{trivial extreme
waves}, for which $\mcs$ is a horizontal line consisting only of
stagnation points and $\psi$ is independent of the $X$ variable.
Unfortunately, it is not known at present whether the extreme waves
we obtain as limits of regular waves are trivial or not.

Nevertheless, it is hoped that Theorem \ref{exi} may be useful in
proofs of the existence of extreme waves in much more general
situations than those in Theorem \ref{zxc}. A key open problem
remains that of determining for what vorticity functions are
(\ref{fhe}) and (\ref{fhc}) necessarily valid for a sequence of
regular waves in the continuum in \cite{CoS}. Theorem \ref{exi}
might also be useful in proving the existence of waves with
stagnation points at the bottom or in the interior of the fluid
domain, in situations when only (\ref{fhe}), but not (\ref{fhc}),
holds for suitable sequences of regular waves.

In Section 5 we address the Stokes conjecture for extreme waves. We
deal with symmetric wave profiles which are locally monotone on
either side of a stagnation point (these assumptions were also
required for the Stokes conjecture in the irrotational case). In
Theorem \ref{sto} we show that at such a stagnation point the
profile has either a corner of $120^\circ$ or a horizontal tangent.
Moreover, we show that the profile necessarily has a corner of
$120^\circ$ whenever the vorticity is nonnegative near the free
surface.

The existence of \emph{trivial extreme waves} shows that the
possibility of a horizontal tangent cannot be ruled out in general.
One should also point out that only smooth vorticity functions are
considered here. For a specific \emph{unbounded} vorticity function,
there exists an explicit example, discovered by Gerstner in 1802,
see \cite[\S 14.40-14.41]{MT}, of an extreme wave whose profile has
cusps at the stagnation points. However, a study of waves with
unbounded vorticity is beyond the scope of this article.

 The proof given here of
the Stokes conjecture for waves with vorticity  is similar in spirit
to that in \cite{AFT} for the irrotational case, in that they are
both based on a blow-up argument, which is a standard tool in the
study of regularity of free boundaries \cite{CafS}. But whilst in
the irrotational case the blow-up is applied in Nekrasov's integral
equation to yield a new integral equation \cite{AFT}, here the
blow-up is performed directly in the physical domain. More
precisely, a \emph{blow-up sequence} (i.e., a sequence of functions
obtained from $\psi$ by rescaling) is shown in Theorem \ref{tblow}
to converge along a subsequence to the solution of a free-boundary
problem for a \emph{harmonic} function in an unbounded domain whose
boundary is curve passing through, and globally monotone on either
side of, the original stagnation point. Apart from a trivial
solution where the free boundary is the real axis, this limiting
problem has another explicit solution, for which the free boundary
consists of two half-line with endpoints at the origin, enclosing an
angle of $120^\circ$ which is symmetric with respect to the
imaginary axis. It was the existence of this solution, nowadays
called the \emph{Stokes corner flow} \cite{Del}, that led Stokes
\cite{S} to his conjecture. It is however the \emph{uniqueness},
which is proved in Theorem \ref{uniq}, of this solution in the class
of symmetric nontrivial solutions of the limiting problem, that
leads to the proof of the conjecture. We show here that the limiting
problem can be described by means of a nonlinear integral equation
for a function $\te^*$ which gives the angle between the tangent to
the free boundary and the horizontal. This equation first arose in
\cite{AFT} as a blow-up limit of Nekrasov's equation, but its
connection to a free-boundary problem seems to have never been
explicitly mentioned in the literature. The monotonicity of the free
boundary means that $0\leq \te^*\leq\pi/2$ on $(0,\infty)$. In this
generality, the uniqueness of the solution of this integral equation
has been proven only very recently in \cite{EV}. Prior to that, a
uniqueness result was known \cite{AFT} only under the restriction
that $0\leq\te^*\leq \pi/3$ on $(0,\infty)$. That result would not
have been enough for a proof of the Stokes conjecture for waves with
vorticity.

We also show, as a byproduct of our approach to the Stokes
conjecture, that if a possibly nonsymmetric extreme wave with
vorticity has lateral tangents at a stagnation point, then the
tangents have to be symmetric with respect to the vertical line
passing through that point and either enclose an angle of
$120^\circ$ or be horizontal. Moreover, the possibility of
horizontal lateral tangents can be ruled out whenever the vorticity
is nonnegative near the free surface.

Some problems left open by the present article are: the structure of
the set of stagnation points for weak solutions of (\ref{apb}), the
regularity of the wave profiles away from stagnation points, the
extent of the validity of (\ref{fhe}) and (\ref{fhc}) for a sequence
in the continuum in \cite{CoS}, the existence of nonsymmetric
extreme waves and the Stokes conjecture in that case, the uniqueness
of solutions of the limiting problem in the absence of symmetry, the
existence of overhanging wave profiles.

\section{Two generalized formulations of the problem}

We consider throughout the rest of the article only the problem of
periodic waves of finite depth. We now make precise the sense in
which (\ref{apb}) is to hold.

It is required throughout that
 \begin{gather}\text{$\mcs$ is locally
rectifiable},\label{rec}\\\psi\in
\Lip(\overline\Om),\label{lip}\\\text{$\mcs$ and $\psi$ are
$2L$-periodic in the horizontal direction},\end{gather} for some
given $L>0$. It is assumed that \be \gamma\in
C^{1,\alpha}([0,B])\quad\text{for some }
\alpha\in(0,1).\label{ga}\ee  It is required that
(\ref{ap*})-(\ref{ap2}) are satisfied in the classical sense. The
condition (\ref{ap0}) is to hold in the following sense:
\be\int_{\Om}\nabla\psi\nabla\zeta\,d\mcl^2=\int_\Om\gamma(\psi)\zeta\,d\mcl^2\label{w0}\quad\text{for
all }\zeta\in C^1_0(\Om),\ee where $\mcl^2$ denotes two-dimensional
Lebesgue measure. Then, standard interior and boundary H\"{o}lder
regularity estimates \cite[Lemma 4.2 and Theorem 6.19]{GT} show that
$\psi\in C^{3,\alpha}_{\textnormal{loc}}(\Om\cup\mcb_F)$, and that
(\ref{ap0}) holds in the classical sense. In particular,
\bese\label{pareq}\begin{align}
\Delta\psix&=-\gamma'(\psi)\psix\quad\text{in }\Om,
\\\Delta\psiy&=-\gamma'(\psi)\psiy\quad\text{in }\Om.\end{align}\ese

Several types of solutions of (\ref{apb}) are described below,
depending on how (\ref{ap3}) is required to hold.

We say that $(\mcs,\mcb_F,\psi,Q)$ is a \emph{classical solution} of
(\ref{apb}) if $\mcs$ is a $C^1$ curve, $\psi\in C^1(\Om\cup\mcs)$
and (\ref{ap3}) holds everywhere on $\mcs$.

We say that $(\mcs,\mcb_F,\psi,Q)$ is a \emph{weak solution} of
(\ref{apb}) if
\begin{align}\int_{\Om}\nabla\psi\nabla\zeta\,d\mcl^2=\int_\Om\gamma(\psi)\zeta\,d\mcl^2-
\int_\mcs(Q-2gY)^{1/2}\zeta\,d\mch^1\label{weak}\\\text{ for all
}\zeta\in C^1_0(\mcu_F),\non\end{align} where
$\mcu_F:=\{(X,Y):X\in\bdr, Y>F\}$ and $\mch^1$ denotes
one-dimensional Hausdorff measure.

 We say that $(\mcs,\mcb_F,\psi,Q)$ is a \emph{Hardy-space solution} of (\ref{apb}) if the
partial derivatives of $\psi$ have non-tangential limits
$\mch^1$-almost everywhere on $\mcs$ which satisfy (\ref{ap3})
$\mch^1$-almost everywhere.

For the definition of a non-tangential limit and for a summary of
notions and results concerning the classical Hardy spaces of
harmonic functions, the reader is referred to the Appendix.

 Obviously, any classical solution of (\ref{apb}) is
both a Hardy-space solution  and a weak solution. The main result of
this section is that the Hardy-space solutions and the weak
solutions of (\ref{apb}) coincide.

\begin{theorem}\label{teqi} Let $(\mcs,\mcb_F,\psi,Q)$ be such that {\rm (\ref{rec})-(\ref{w0})}
 hold. Then
$(\mcs,\mcb_F,\psi,$\\$Q)$ is a Hardy-space solution of {\rm
(\ref{apb})} if and only if it is a weak solution.
\end{theorem}

The proof of Theorem \ref{teqi} follows from a series of results
concerning some properties of solutions $(\mcs,\mcb_F,\psi,Q)$ of
(\ref{rec})-(\ref{w0}).

 In the irrotational case, the
 partial derivatives of $\psi$ are harmonic functions, and their boundedness in $\Om$ ensures,
 by Fatou's Theorem, that they
 have non-tangential limits $\mch^1$-almost everywhere on
$\mcs$. Here this result is extended to the general case of waves
with vorticity.

\begin{proposition}\label{weha} Let $(\mcs,\mcb_F,\psi,Q)$ be such that {\rm (\ref{rec})-(\ref{w0})}
 hold. Then
the partial derivatives of $\psi$ have non-tangential limits
$\mch^1$-almost everywhere on $\mcs$.
\end{proposition}

 The proof of Proposition \ref{weha} is based on the following
 simple observation,
  whose conclusion holds more generally.

\begin{lemma}\label{ntg} Let $\mcg\subset\bdr^2$ be a bounded open set
 whose boundary is a
rectifiable Jordan curve $\mathcal{J}$. Let $w\in
C^{2,\alpha}_{\textnormal{loc}}(\mcg)\cap L^\infty(\mcg)$ be such
that
\[\Delta w=q\quad\text{in }\mcg,\]
where $q\in C^{0,\alpha}_\loc(\mcg)\cap L^\infty(\mcg)$. Then $w$
has non-tangential limits $\mch^1$-almost everywhere on
$\mathcal{J}$.
\end{lemma}

\begin{proof}[Proof of Lemma \ref{ntg}] Let us write $w=u+v$, where
$u$ is the Newtonian potential of $q$,
\[ u(x)=\frac{1}{2\pi}\int_\mcg\log|x-y|q(y)\,d\mcl^2(y)\quad\text{for all }x\in\bdr ^2.\]
It is well known \cite[Lemma 4.1 and Lemma 4.2]{GT} that $u\in
C^1(\bdr^2)\cap C^{2,\alpha}_\loc(\mcg)$ satisfies
\[\Delta u=q\quad\text{in }\mcg.\]
Hence $v$ is a bounded harmonic function in $\mcg$, and therefore
has non-tangential limits $\mch^1$-almost everywhere on
$\mathcal{J}$. Since $u$ is continuous on $\bdr^2$, the required
conclusion follows.
\end{proof}

\begin{proof}[Proof of Proposition \ref{weha}] It suffices to apply
Lemma \ref{ntg} with the partial derivatives of $\psi$, which
satisfy (\ref{pareq}), in the role of $w$ in an obvious domain
$\mcg$.
\end{proof}

Under the assumptions of Proposition \ref{weha} let, for
$\mch^1$-almost every $(X_0, Y_0)\in\mcs$, \be
\nabla\psi(X_0,Y_0):=\lim_{(X,Y)\to(X_0,
Y_0)}\nabla\psi(X,Y),\label{gra}\ee where the limit is taken
non-tangentially within $\Om$. For $\mch^1$-almost every $(X_0,
Y_0)$ $\in\mcs$, let \be\frac{\partial\psi}{\partial
n}(X_0,Y_0):=\nabla\psi(X_0,Y_0)\cdot {\bf n}(X_0,Y_0),\label{no}\ee
where $\cdot$ denotes the standard inner product in $\bdr^2$ and
${\bf n}(X_0,Y_0)$ is the outward unit normal to $\Om$ at
$(X_0,Y_0)$.

\begin{proposition}\label{noco} Let $(\mcs,\mcb_F,\psi,Q)$ be such that
{\rm (\ref{rec})-(\ref{w0})}
 hold, and suppose in addition that {\rm (\ref{ap2})} is satisfied.
  Then, in the
notation of {\rm(\ref{gra})} and {\rm (\ref{no})}, $\psi$ satisfies
{\rm(\ref{ap3})} $\mch^1$-almost everywhere on $\mcs$ if and only if
\[ \frac{\partial\psi}{\partial n}(X,Y)=-(Q-2gY)^{1/2}\quad\text{ for
$\mch^1$-almost every $(X,Y)\in\mcs$.}\]
\end{proposition}

The proof of Proposition \ref{noco} depends on the following lemma.

\begin{lemma}\label{tg} Let $\mcg\subset\bdr^2$ be a bounded open set whose boundary is a
rectifiable Jordan curve $\mathcal{J}$. Let $w\in
C^1(\mcg)\cap\Lip(\overline \mcg)$ be such that the partial
derivatives of $w$ have non-tangential limits $\mch^1$-almost
everywhere on $\mathcal{J}$. Suppose that $w$ is a constant on a
closed arc $\mathcal{I}$ of $\mathcal{J}$. Then \[ \nabla w(X_0,
Y_0)\cdot{\bf t}(X_0,Y_0)=0\quad\text{for $\mch^1$-almost every
$(X_0,Y_0)\in\mathcal{I}$},\] where $\nabla w(X_0,Y_0)$ denotes the
non-tangential limit within $\mcg$ of $\nabla w$ at $(X_0,Y_0)$ and
${\bf t}(X_0, Y_0)$ is a unit tangent to $\mathcal{J}$ at $(X_0,
Y_0)$.
\end{lemma}

\begin{proof}[Proof of Lemma \ref{tg}] Let $\mcd$ be the unit disc in
the plane, and let $f:\mcd\to \mcg$ be a conformal mapping from
$\mcd$ onto $\mcg$. Since the boundary of $\mcg$ is a rectifiable
Jordan curve, it is classical \cite[Theorem 3.11 and Theorem
3.12]{Du} that $f$ is a homeomorphism from the closure of $\mcd$
onto the closure of $\mcg$, $f'$ belongs to the Hardy space
$H^1_\bdc(\mcd)$, the mapping $t\mapsto f(e^{it})$ is locally
absolutely continuous and \[ \frac{d}{dt}f(e^{it})=\lim_{r\nearrow
1}ire^{it}f'(re^{it})\quad\text{for almost every }t\in\bdr,\] where
$'$ denotes complex differentiation. Let $a,\, b\in\bdr$ be such
that $t\mapsto f(e^{it})$ is a bijection from $[a,b]$ onto
$\mathcal{I}$. Then, for every $t_1,\,t_2\in[a,b]$ with $t_1\leq
t_2$ and for every $r\in(0,1)$, \be
w(f(re^{it_2}))-w(f(re^{it_1}))=\int_{t_1}^{t_2}\nabla
w(f(re^{it}))\cdot \frac{d}{dt} f(re^{it})\,dt. \label{fuca}\ee We
now pass to the limit as $r\nearrow 1$ in (\ref{fuca}) using the
Dominated Convergence Theorem, with the integrands bounded in
absolute value by the integrable function $||\nabla
w||_{L^\infty(\mcg)} M_{\textnormal{rad}}[f']$, where
$M_{\textnormal{rad}}[f']$ denotes the radial maximal function, see
the Appendix, of the function $f'\in H^1_{\bdc}(\mcd)$, to obtain
(\ref{fuca}) with $r=1$. It is important in this argument that, for
almost every $t\in (a,b)$, $f(re^{it})\to f(e^{it})$
non-tangentially within $\mcg$ as $r\nearrow 1$, see \cite[Section
3.5]{Du}. Since $\frac{d}{dt} f(e^{it})\neq 0$ for almost every
$t\in(a,b)$, the required conclusion follows.
\end{proof}

\begin{proof}[Proof of Proposition \ref{noco}] The required result
follows immediately by applying Lemma \ref{tg} to the function
$\psi$ in an obvious domain $\mcg$. Note also that, when it is
assumed that $\psi$ satisfies (\ref{ap3}), the sign of the normal
derivative of $\psi$ can be determined from the fact that $\psi=0$
on $\mcs$ and $\psi\geq 0$ in $\Om$.
\end{proof}

\begin{proposition}\label{green} Let $(\mcs,\mcb_F,\psi,Q)$
 be such that {\rm (\ref{rec})-(\ref{w0})}
 hold. Then, in the notation of
 {\rm (\ref{no})},
\be\int_{\Om}\nabla\psi\nabla\zeta\,d\mcl^2=\int_\Om\gamma(\psi)\zeta\,d\mcl^2+
\int_\mcs \frac{\partial\psi}{\partial
n}\,\zeta\,d\mch^1\label{gree}\ee for all $\zeta\in C_0^1(\mcu_F)$.
\end{proposition}

\begin{proof}[Proof of Proposition \ref{green}] Fix
$\zeta\in C_0^1(\mcu_F)$. Then one can find points $Z_1,\,Z_2$ on
$\mcs$, $W_1,\,W_2$ on $\mcb_F$, and a bounded open set $\mcg$
contained in $\Om$, whose boundary is a rectifiable Jordan curve
$\mathcal{J}:=\mathcal{I}\cup\mathcal{L}_2\cup\mathcal{M}\cup\mathcal{L}_1$,
such that
\begin{gather} (\text{supp}\,\zeta) \cap\Om\subset
\mcg,\label{mna}\\\text{dist}\,((\text{supp}\,\zeta),
\mathcal{J}\setminus\mathcal{I})>0,\label{mnb}\end{gather} where
$\mathcal{I}$ is the arc of $\mcs$ joining $Z_1$ and $Z_2$,
$\mathcal{L}_2$ is an arc contained in $\Om$ joining $Z_2$ and
$W_2$, $\mathcal{M}$ is the line segment joining $W_2$ and $W_1$,
and $\mathcal{L}_1$ is an arc contained in $\Om$ joining $W_1$ and
$Z_1$. To prove (\ref{gree}) is equivalent, by means of
(\ref{mna})-(\ref{mnb}), to proving \be
\int_\mcg\nabla\psi\nabla\zeta\,d\mcl^2=\int_\mcg\gamma(\psi)\zeta\,d\mcl^2+
\int_{\mathcal{I}} \frac{\partial\psi}{\partial
n}\,\zeta\,d\mch^1.\label{negr}\ee

Let $\mcd$ be the unit disc in the plane, and let $f:\mcd\to \mcg$
be a conformal mapping from $\mcd$ onto $\mcg$ and a homeomorphism
from the closure of $\mcd$ onto the closure of $\mcg$. Let $a,\,
b\in\bdr$ be such that $t\mapsto f(e^{it})$ is a bijection from
$[a,b]$ onto $\mathcal{I}$. For every $r\in(0,1)$, let $\mcd_r$ be
the disc centred at $0$ and of radius $r$, and $\mcg_r:=f(\mcd_r)$.
It follows from (\ref{mnb}) and the standard Green's Formula that,
for all $r$ sufficiently close to $1$, \be
\int_{\mcg_r}\nabla\psi\nabla\zeta\,d\mcl^2=\int_{\mcg_r}\gamma(\psi)\zeta\,d\mcl^2+
\int_a^b\left[\nabla\psi(f(re^{it}))\cdot
\left(i\frac{d}{dt}f(re^{it})\right)\right]\zeta(f(re^{it}))\,dt\label{ave}\ee
Since $f(re^{it})\to f(e^{it})$ non-tangentially within $\mcg$ as
$r\nearrow 1$, for almost every $t\in (a,b)$, and since the
integrands in the last term of (\ref{ave}) are bounded in absolute
value by the integrable function $||\nabla
w||_{L^\infty(\mcg)}||\zeta||_{L^\infty(\mcg)}
M_{\textnormal{rad}}[f']$, one can pass to the limit as $r\nearrow1
$ in (\ref{ave}), using the Dominated Convergence Theorem, to get
(\ref{negr}). This completes the proof of Proposition \ref{green}.
\end{proof}

\begin{proof}[Proof of Theorem \ref{teqi}] Suppose first that
$(\mcs,\mcb_F,\psi,Q)$ is a Hardy-space solution of (\ref{apb}). It
is immediate from Proposition \ref{noco} and Proposition \ref{green}
that $(\mcs,\mcb_F,\psi,Q)$ is a weak solution.

Suppose now that $(\mcs,\mcb_F,\psi,Q)$ is a weak solution of
(\ref{apb}). By comparing (\ref{weak}) and (\ref{gree}), we deduce
that, for all $\zeta\in C_0^1(\mcu_F)$,
\be\int_\mcs\left[\frac{\partial\psi}{\partial
n}+(Q-2gY)^{1/2}\right]\zeta\,d\mch^1=0. \label{cali}\ee A simple
approximation argument shows that (\ref{cali}) also holds for all
$\zeta\in C_0(\mcu_F)$, from where it is immediate that
\[ \frac{\partial\psi}{\partial
n}(X,Y)=-(Q-2gY)^{1/2}\quad\text{ for $\mch^1$-almost every
$(X,Y)\in\mcs$.}\] It follows from Proposition \ref{noco} that
$(\mcs,\mcb_F,\psi,Q)$ is a Hardy-space solution.

The proof of Theorem \ref{teqi} is therefore completed.
\end{proof}

We conclude this section with the following obvious observation.

\begin{proposition}\label{tran} Let $(\mcs,\mcb_F, \psi, Q)$ be a classical/weak
solution of {\rm (\ref{apb})}, $\Om$ be the open set whose boundary
consists of $\mcs$ and $\mcb_F$, and $G\in\bdr$. Let
$\widehat\Om:=\{(X,Y)\in\bdr^2: (X,Y+G)\in\Om \}$, and
$\hat\psi:\widehat\Om\to\bdr$ be given by
$\hat\psi(X,Y):=\psi(X,Y+G)$ for all $(X,Y)\in\widehat\Om$. Then the
boundary of $\widehat\Om$ consists of the line $\mcb_{F-G}$ and a
curve $\widehat\mcs$, and $(\widehat\mcs,\mcb_{F-G}, \hat\psi,
Q-2gG)$ is a classical/weak solution of {\rm (\ref{apb})}.
\end{proposition}

\section{An a priori estimate on the pressure in the fluid}

In this section we use the maximum principle to derive an a priori
estimate
 on the pressure in the fluid.
Apart from being of interest in itself, this result plays an
essential role in the investigation of the existence of extreme
waves and the Stokes conjecture with vorticity.

 Let $(\mcs,\mcb_F, \psi, Q)$ be a classical solution of {\rm (\ref{apb})}. Let $\hg :[0,B]\to\bdr$ be given by \be
\hg(r)=\int_0^r\gamma(t)\,dt\quad\text{for all
}r\in[0,B].\label{hg}\ee The function $R[\psi]$ given in $\ovo$ by
\[
R[\psi]:=\frac{1}{2}\vert\nabla\psi\vert^{2}+gY-\frac{1}{2}Q+\hg(\psi)\]
is, up to a constant, the negative of the pressure in the fluid. Let
 $T[\psi]$ be given in $\ovo$
by \[
T[\psi]:=\frac{1}{2}\vert\nabla\psi\vert^{2}+gy-\frac{1}{2}Q+\hg(\psi)-\varpi\psi,
\] where \[\varpi:=\frac{1}{2}\max\left\{0,
\max_{r\in[0,B]}\gamma(r)\right\}.\] Obviously $T[\psi]=R[\psi]$
whenever $\gamma(r)\leq 0$ for all $r\in[0,B]$. Theorem \ref{bxv}
below is an extension of {\rm \cite[Theorem 2.1 and Theorem
2.4]{EV4}}, where the same result was proved under the assumption
that $\gamma:[0,B]\to\bdr$ does not change sign.

\begin{theorem}\label{bxv} Let $(\mcs,\mcb_F, \psi, Q)$ be a classical solution of {\rm
(\ref{apb})} such that $\psiy<0$ in $\Om$.
 Then $T[\psi]\leq 0$ in $\ovo$.
\end{theorem}

\begin{proof}[Proof of Theorem \ref{bxv}] The proof is merely an
application of a result in Sperb \cite[Section 5.2]{Sp}. I am
grateful to John Toland for pointing out this reference to me.

 The required result is obtained from a more general one. Let
$\lambda:[0,B]\to\bdr$ be a $C^1$ function, let
$\Lambda:[0,B]\to\bdr$ be given by $\Lambda(r)=\int_0^r\la(t)\,dt$
for all $r\in[0,B]$, and let $S:\ovo\to\bdr$ be given by \be
S:=\frac{1}{2}\vert\nabla\psi\vert^{2}+gY-\frac{1}{2}Q+\hg(\psi)+\Lambda(\psi).\label{r}\ee
Then $S=0$ on $\mcs$. We seek conditions on $\la$ which ensure that
$S\leq 0$ in $\ovo$.

 Let $W:\ovo\to\bdr$ be given by \be
W=\frac{1}{2}\vert\nabla\psi\vert^{2}+\hg(\psi)+\Lambda(\psi).\label{ff}\ee
It is easy to check, using the fact that (\ref{ap0}) holds, that $W$
satisfies the following elliptic equation in $\Om$: \be \Delta
W+\frac{L_1}{|\nabla \psi|^2}W_{_X}+\frac{L_2}{|\nabla
\psi|^2}W_{_Y}=\lambda'(\psi)|\nabla\psi|^2+(2\lambda(\psi)+\gamma(\psi))\lambda(\psi),\label{ele}\ee
where \be L_1:=-2[W_{_X}-(2\lambda(\psi)+\gamma(\psi))\psix],\quad
L_2:=-2[W_{_Y}-(2\lambda(\psi)+\gamma(\psi))\psiy].\label{coff}\ee
Equation (\ref{ele}) is \cite[equation (5.17), p.\ 69]{Sp}, which is
correct, despite the fact that there is a  misprint in
\cite[equation (5.16), p.\ 69]{Sp}.  Note also that \be
W_{_Y}=\lambda(\psi)\psiy\quad\text{on }\mcb_F.\label{bco}\ee It is
immediate from (\ref{ele})-(\ref{bco}) that \begin{align} &\Delta
S+\frac{M_1}{|\nabla \psi|^2}S_{_X}+\frac{M_2}{|\nabla
\psi|^2}S_{_Y}\non\\&=
\lambda'(\psi)|\nabla\psi|^2+(2\lambda(\psi)+\gamma(\psi))\lambda(\psi)
+\frac{2g}{|\nabla\psi|^2}[g+(2\lambda(\psi)+\gamma(\psi))\psiy],\label{elef}
\end{align}
and \be S_{_Y}=g+\lambda(\psi)\psiy\quad\text{on
}\mcb_F,\label{bco1}\ee
where \begin{align} M_1:&=-2[S_{_X}-(2\lambda(\psi)+\gamma(\psi))\psix],\non\\
M_2:&=-2[S_{_Y}-2g-(2\lambda(\psi)+\gamma(\psi))\psiy].\label{bco2}\end{align}
Since $S=0$ on $\mcs$ and $\psiy<0$ in $\Om$, the maximum principle
shows that $S\leq 0$ in $\ovo$ whenever \be  \la(r)\leq
0,\,\,2\la(r)+\gamma(r)\leq 0,\,\, \la'(r)\geq 0\quad\text{for all
}r\in[0,B]. \ee In particular, $T[\psi]\leq 0$ in $\ovo$. This
completes the proof of Theorem \ref{bxv}.
\end{proof}

Let us also record here the following immediate consequence of
(\ref{ele})-(\ref{bco}) with $\la\equiv 0$.

\begin{proposition}\label{nqt}
Let $(\mcs,\mcb_F, \psi, Q)$ be a classical solution of {\rm
(\ref{apb})} such that $|\nabla\psi|\neq 0$ in $\Om\cup\mcb_F$. Then
\[ \min_\mcs |\nabla\psi|^2\leq
|\nabla\psi(X,Y))|^2+2\hg(\psi(X,Y))\leq\max_\mcs
|\nabla\psi|^2\quad\text{for all }(X,Y)\in\ovo.\]
\end{proposition}

\begin{remark} {\rm The estimate in Proposition \ref{nqt}
 holds with equalities for any solution of (\ref{apb})
  for which $\mcs$ is a horizontal line and $\psi$ does not depend on $X$.}
\end{remark}

\section{On the existence of extreme waves}

Let $(\mcs,\mcb_F, \psi,Q)$ be a weak solution of (\ref{apb}). We
say that a point $(X_0,Y_0)$ on $\mcs$ is a \emph{stagnation point}
if $Q-2gY_0=0$. This would formally correspond to the fact that
$\nabla\psi(X_0,Y_0)=(0,0)$. A weak solution of (\ref{apb}) with
stagnation points on the free surface $\mcs$ is called an
\emph{extreme wave}.

In view of Proposition \ref{tran}, there is no loss of generality in
considering only solutions of (\ref{apb}) for which $F=0$. In this
section we are interested in solutions $(\mcs, \mcb_0, \psi,Q)$ of
(\ref{apb}) for which
 \be \psi \text{ is even in the $X$ variable},\qquad\qquad \psiy< 0\text{ in }\Om.
 \label{ex2}\ee
and, in some situations, also
\begin{align}\mcs=\{(X,\eta(X)):&X\in \bdr\}, \text{with}\label{e1}\\ &\text{$\eta\in
C^1(\bdr)$, $2L$-periodic, even, and $\eta'< 0$ on
$(0,L)$.}\non\end{align}

The following result gives general conditions under which a sequence
of regular waves contains a subsequence converging in a certain
sense to an extreme wave. Here and in what follows, for any (weak)
solution $(\mcs,\mcb_0, \psi,Q)$ of (\ref{apb}), we extend $\psi$ to
$\bdr^2_+$ with the value $0$ in $\bdr^2_+\setminus\overline\Om$.
The extension, denoted also by $\psi$, is a Lipschitz function on
$\bdr^2_+$.

\begin{theorem}\label{exi}
 Let
$\{(\mcs_j,\mcb_0,\psi^j,Q_j)\}_{j\geq 1}$ be a sequence of
classical solutions of {\rm (\ref{apb})} for which {\rm (\ref{ex2})}
and { \rm (\ref{e1})} hold.  Suppose that \be \text{the sequence
$\{Q_j\}_{j\geq 1}$ is bounded above}\label{e6}. \ee Then there
exists a weak solution $(\widetilde\mcs,\mcb_0, \tips, \widetilde
Q)$ of {\rm(\ref{apb})} for which {\rm (\ref{ex2})} holds, such
that, along a subsequence (not relabeled),
\begin{align}
&Q_j\to\widetilde Q,\label{f1}\\
&\psi^j\to\tips\quad\text{uniformly on }\bdr^2_+,\label{f2}\\
&\nabla\psi^j\to\nabla\tips\quad\text{weak* in
}L^\infty(\bdr^2_+).\label{f3}\end{align}
 If, in addition,
\be\text{$|\nabla\psi^j(0,\eta_j(0))|\to 0$ as
$j\to\infty$},\label{e5}\ee then $(\widetilde\mcs,\mcb_0,\tips,
\widetilde Q)$ is an extreme wave.
\end{theorem}

\begin{remark}{\rm The proof of Theorem \ref{exi} also provides a precise sense in which
 the sequence of curves $\{\mcs_j \}\indj$ converges along a subsequence to $\wids$. For the
 sake of brevity, we have chosen not to include this in the
 statement of the theorem.}
\end{remark}

\begin{remark}
{\rm  The proof of Theorem \ref{exi} would be simpler if it were
assumed that \[\text{the family }\{\eta_j\}_{j\geq 1}\text{ is
equi-Lipschitz on $\bdr$}.\] Such an assumption would probably be
difficult to verify in practice, so it is important that we do not
need it in Theorem \ref{exi}.}
\end{remark}

 Constantin and Strauss \cite[Theorem 1.1]{CoS} proved that, if
 $\gamma:[0,B]\to\bdr$ satisfies the condition
\be
\int_0^B\left[\frac{\pi^2(B-r)^2}{L^2}(2\hg_{\textnormal{max}}-2\hg(r))^{1/2}+
(2\hg_{\textnormal{max}}-2\hg(r))^{3/2}
\right]\,dr<gB^2,\label{jhb}\ee  where $\hg$ is given by (\ref{hg})
and $\hg_{\textnormal{max}}:=\max_{r\in[0,B]}\hg(r)$, then there
exists a set $\mcc$ (connected in a certain function space) of
solutions of (\ref{apb}) of the form $(\mcs,\mcb_0, \psi, Q)$,
satisfying
 (\ref{ex2}) and (\ref{e1}), which contains a sequence $\{(\mcs_j,\mcb_0,
 \psi^j,$ $Q_j)\}_{j\geq 1}$
  such that $\max_{\overline\Om_j}\psiy^j\to 0$ as $j\to\infty$.

The following new result concerning the convergence of a sequence of
regular waves in $\mcc$ to an extreme wave is easily obtained by
combining Theorem \ref{exi} with existing results in literature on
the validity of (\ref{e6}) and (\ref{e5}) for a sequence in~$\mcc$.

\begin{theorem}\label{zxc} Let $\gamma:[0,B]\to\bdr$ be such that {\rm (\ref{jhb})} holds,
and suppose in addition that \[ \text{$\gamma(0)<0$, $\gamma(r)\leq
0$ and $\gamma'(r)\geq 0$ for all $r\in[0,B]$.}\]
 Let $\{(\mcs_j,\mcb_0,\psi^j,Q_j)\}_{j\geq 1}$ be a sequence
in $\mcc$ such that \be \text{$\max_{\overline\Om_j}\psiy^j\to 0$ as
$j\to\infty$}.\label{vmz}\ee
 Then $\{(\mcs_j,\mcb_0,\psi^j,Q_j)\}_{j\geq 1}$ converges in
 the sense of Theorem \ref{exi}, along a subsequence, to an extreme
 wave $(\wids, \mcb_0,\tips, \widetilde Q)$.

\end{theorem}

\begin{remark}\label{twq}{\rm
Let $\gamma:[0,B]\to\bdr$ be such that $\gamma(0)<0$ and
$\gamma(r)\leq 0$ for all $r\in[0, B]$. Let $\Upsilon:[0,B]\to\bdr$
be given by
\[\Upsilon(r)=\int_0^r\frac{1}{{(-2\hg(t))^{1/2}}}\,dt\quad\text{for all }r\in[0,B],\]
where the function $\hg$ is given by (\ref{hg}). Then $\Upsilon$ is
a bijection from $[0,B]$ onto $[0,\Upsilon(B)]$, with inverse
$\Upsilon^{-1}:[0,\Upsilon(B)]\to[0,B]$. Let $\mcs:=\{(X,\Upsilon
(B)):X\in\bdr\}$ and $\Om$ be the strip whose boundary consists of
$\mcs$ and $\mcb_0$. Let $Q:=2g\Upsilon(B)$ and $\psi:\Om\to\bdr$ be
given by
\[\psi(X,Y):=\Upsilon^{-1}(\Upsilon(B)-Y)\quad\text{for all }(X,Y)\in\Om.\]
It is easy to check that $(\mcs,\mcb_0, \psi, Q)$ is a solution of
(\ref{apb}) for which all the points of $\mcs$ are stagnation
points. We call such a solution of (\ref{apb}) a \emph{trivial
extreme wave}.}
\end{remark}

\begin{remark} {\rm It is not known whether the extreme wave obtained in
Theorem \ref{zxc} is trivial or not. It is natural to conjecture
that the conclusion of Theorem \ref{zxc} remains valid in the
absence of the condition $\gamma(0)<0$, but the present method of
proof of Theorem \ref{zxc} cannot handle this more general
situation. The difficulty is to prove the validity of (\ref{e6}) for
a suitable sequence in $\mcc$. This fact \emph{can} be proved in the
irrotational case, thus leading to the existence of a
\emph{nontrivial} extreme wave, but the only proof we know makes use
of Nekrasov's integral equation, and because this method cannot be
used for rotational waves we refrain from giving any details here.}
\end{remark}

We now give the proof of Theorem \ref{exi}, and then that of Theorem
\ref{zxc}.

\begin{proof}[Proof of Theorem \ref{exi}]

Let $\{(\mcs_j, \mcb_0, \psi^j,Q_j)\}_{j\geq 1}$ be as in the
statement of the theorem, with $\mcs_j=\{(X,\eta_j(X)):X\in\bdr\}$
for all $j\geq 1$.  The condition (\ref{e6}) means that
\begin{gather} \text{the sequence $\{\max_{\mcs_j}|\nabla\psi^j|\}_{j\geq 1}$
is uniformly bounded above,}\label{moiu} \\\text{the sequence
$\{\eta_j\}_{j\geq 1}$ is uniformly bounded
above.}\label{oiu}\end{gather} It follows from (\ref{moiu}) and
Proposition \ref{nqt} that \be\text{the family $\{\psi_j\}_{j\geq
1}$ is equi-Lipschitz on $\bdr^2_+$.}\label{eqil}\ee We deduce from
(\ref{eqil}), using (\ref{e1}) and the relation
\[-B=\psi^j(L,\eta_j(L))-\psi^j(L,0)=\int_0^{\eta_j(L)}\psiy(L,V)\,dV,\]
that \be\text{the sequence $\{\eta_j\}_{j\geq 1}$ is uniformly
bounded away from $0$.}\label{poiu}\ee This implies that
\be\text{$\{Q_j\}\indj$ is bounded away from $0$.}\ee

Let $\ell_j$ denote the length of the arc
$\{(X,\eta_j(X)):X\in[0,L]\}$, for all $j\geq 1$. It follows from
(\ref{e1}) and (\ref{oiu}) that \be\text{$\{\ell_j\}_{j\geq 1}$ is
bounded above and away from $0$.}\label{bdl}\ee For any $j\geq 1$, a
parametrization of the curve $\mcs_j$ is given by
$\mcs_j=\{(u_j(s),v_j(s)):s\in\bdr\}$, where $u_j,\,
v_j:\bdr\to\bdr$ are $C^1$ functions, periodic of period $1$, such
that
\begin{gather}
\text{ $u_j(0)=0$, $v_j(0)=\eta_j(0)$,\quad $u_j(1)=L$,
$v_j(1)=\eta_j(L),$}\label{ep}\\
u_j\text{ is odd,}\qquad v_j\text{ is even,}\\
 u_j'(s)\geq 0\text{ for all } s\in\bdr,\qquad
v_j'(s)\leq 0\text{ for all
}s\in[0,1],\label{vj}\\u_j'(s)^2+v_j'(s)^2=\ell_j^2\quad\text{for
all }s\in\bdr.\label{arcl}\end{gather}

It follows from (\ref{e6}) and (\ref{oiu})-(\ref{arcl}) that there
exist constants $\widetilde Q>0$, $\tilde\ell>0$, and functions
$\tips\in\Lip(\bdr^2_+)$ and $\tilde u,\,\tilde v\in\Lip(\bdr)$,
with $\tilde u,\,\tilde v$ periodic of period $1$, such that, along
a subsequence (not relabeled), (\ref{f1})-(\ref{f3}) hold and
\begin{align}
& \ell_j\to \tilde\ell,\label{f11}\\
&u_j\to\tilde u,\, v_j\to\tilde v \quad\text{uniformly on }\bdr,\label{f4}\\
&u_j'\to\tilde u',\, v_j'\to\tilde v'\quad\text{weak* in
}L^\infty(\bdr).\label{f5}
\end{align}
It is immediate from (\ref{ep})-(\ref{arcl}) that \begin{gather}
\text{$\tilde u(0)=0$,\quad $\tilde u(1)=L$,}\label{nb}\\
\tilde u\text{ is odd,}\qquad \tilde v\text{ is even,}\\
\tilde u'(s)\geq 0\text{ for a.e.\ }s\in\bdr,\qquad\tilde v'(s)\leq
0\text{ for a.e.\ }s\in(0,1),\label{axz}\\\quad \tilde
u'(s)^2+\tilde v'(s)^2\leq \tilde\ell^2\quad\text{for almost every
}s\in\bdr.\label{papa}\end{gather} It is also a consequence of
(\ref{vj}) and (\ref{arcl})
 that, for all $j\geq 1$ and for every $a,\,b\in[0,1]$
with $a<b$, \[
(b-a)\ell_j=\int_a^b(u_j'(s)^2+v_j'(s)^2)^{1/2}\,ds\leq|u_j(b)-u_j(a)|+|v_j(b)-v_j(a)|.\]
This implies that \be (b-a)\tilde\ell\leq|\tilde u(b)-\tilde
u(a)|+|\tilde v(b)-\tilde v(a)|.\label{taa}\ee Therefore,
\begin{gather}\text{the mapping $s\mapsto(\tilde u(s),\tilde v(s))$ is
injective on $[0,1]$,}\label{lij}\\
\tilde\ell\leq|\tilde u'(s)|+|\tilde v'(s)|\quad\text{for almost
every }s\in(0,1).\label{nond}
\end{gather}

We would now like to prove that \be\text{the mapping
$s\mapsto(\tilde u(s),\tilde v(s))$ is injective on
$\bdr$.}\label{fbi}\ee
 Let $\sigma\in [0,1)$ and
$\varsigma\in(0,1]$, with $\sigma<\varsigma$, be given by
\bese\label{epr}\begin{align}\sigma:&=\max\{s\in[0,1]: \tilde
u(s)=0\},\\\varsigma:&=\min\{s\in[0,1]: \tilde
u(s)=L\}.\end{align}\ese To prove (\ref{fbi}) it suffices, in view
of (\ref{lij}) and (\ref{nb})-(\ref{axz}), to show that $\sigma=0$
and $\varsigma=1$. Note from (\ref{axz}), (\ref{papa}), (\ref{nond})
and (\ref{epr}) that \be \tilde u'(s)=0,\quad \tilde
v'(s)=-\tilde\ell\quad\text{for almost every }s\in
 (0,\sigma)\cup (\varsigma,1).\label{dera}\ee
Let $I:=\cup_{n\in\mathbb{Z}}(2n-\sigma,2n+\sigma)$,
$J:=\cup_{n\in\mathbb{Z}}(2n+1-\varsigma,2n+1+\varsigma)$, and
\begin{gather}\mathcal{I}:=\{(\tilde
u(s),\tilde v(s)):s\in I\},\quad\mathcal{J}:=\{(\tilde u(s),\tilde
v(s)):s\in J\},\non\\\wids:=\{(\tilde u(s),\tilde
v(s)):s\in\bdr\setminus (I\cup J)\}.\non
\end{gather}
Then each of $\mathcal{I}$ and $\mathcal{J}$ is either empty or a
countable union of half-open vertical segments, while $\wids$ is a
locally rectifiable curve, $2L$-periodic in the horizontal direction
and symmetric.  Let $\wido$ be the domain whose boundary consists of
$\widetilde\mcs$ and $\mcb_0$. We first show that $\mathcal{I}$ and
$\mathcal{J}$ are empty, and then that $(\widetilde\mcs, \mcb_0,
\tips, \widetilde Q)$ is a weak solution of (\ref{apb}).

It is immediate from (\ref{f4}) that, for any compact set
$\mathcal{K}\subset\bdr^2$,
\begin{gather}
\text{$\mathcal{K}\subset (\wido\setminus \mathcal{J})\cup\mcb$
implies $\mathcal{K}\subset\Om_j\cup\mcb$ for all $j$ sufficiently large,}\label{co1}\\
\begin{aligned}\text{$\mathcal{K}\subset
\bdr^2_+\setminus(\wido\cup\wids\cup\mathcal{I})$}&\text{ implies}\label{co2}\\
&\text{ $\mathcal{K}\subset\bdr^2_+\setminus(\Om_j\cup\mcs_j)$ for
all $j$ sufficiently large.}\end{aligned}
\end{gather}
It is obvious that $0\leq\tips\leq B$ in $\wido$ and that $\tips=B$
on $\mcb_0$. Also, it follows from (\ref{co2}) that $\tips=0$ in
$\bdr^2_+\setminus(\wido\cup\wids\cup\mathcal{I})$, and hence,
 using the continuity of $\tips$ on $\bdr^2_+$, that $\tips=0$ on
$\bdr^2_+\setminus\wido$. Now, for every $j\geq 1$, (\ref{weak}) can
be written in the form
\begin{align}\int_{\Om_j}\nabla\psi^j\nabla\zeta\,d\mcl^2=\int_{\Om_j}\gamma(\psi^j)\zeta\,d\mcl^2-
\int_\bdr(Q_j&-2gv_j(s))^{1/2}\zeta(u_j(s),v_j(s))\ell_j\,ds\non\\&\text{
for all }\zeta\in C^1_0(\bdr^2_+).\label{wej}\end{align} The
validity of (\ref{f1})-(\ref{f3}), (\ref{f11}), (\ref{f4}),
(\ref{co1}) and (\ref{co2}) makes it possible to pass to the limit
as $j\to\infty$ in (\ref{wej}), to obtain
\begin{align}\int_{\wido}\nabla\tips\nabla\zeta\,d\mcl^2=
\int_{\wido}\gamma(\tips)\zeta\,d\mcl^2- \int_\bdr(\widetilde
Q&-2g\tilde v(s))^{1/2}\zeta(\tilde u(s),\tilde
v(s))\tilde\ell\,ds\non\\&\text{ for all }\zeta\in
C^1_0(\bdr^2_+).\label{weil}\end{align}

With $\sigma$ defined in (\ref{epr}), we now claim that
 $\sigma=0$. Suppose for a contradiction that
this is not so. Let $\mathcal{\mcd}$ be the disc centred at $(0,
\tilde v(0))$
 and with the point $(0, \tilde v(\sigma))$ on its boundary.
 It follows from (\ref{weil}) and (\ref{dera}) that
 \[ \int_{\tilde v(\sigma)}^{\tilde v(0)}(\widetilde Q-2gY)^{1/2}\zeta(0,Y)\,dY=0\quad\text{for
 all $\zeta\in C^1_0(\mathcal{\mcd})$}.\]
Since this is clearly not possible, it follows that $\sigma=0$.

With $\varsigma$ defined in (\ref{epr}), we now claim that
$\varsigma=1$.
 Suppose for a
contradiction that this is not so. Let \[\mathcal{R}:=\{(X,Y)\in
\bdr^2:-L<X<3L, 0<Y<\tilde v(\varsigma)\},\] and let
$\mathcal{R}_-:=\{(X,Y)\in\mathcal{R}: X<L\}$,
$\mathcal{R}_+:=\{(X,Y)\in\mathcal{R}: X>L\}$ and
$\mathcal{R}_L:=\{(X,Y)\in\mathcal{R}: X=L\}$.
 It follows from
(\ref{weil}) and (\ref{dera}) that, for all $\zeta\in \Lip_0(\mcr)$,
\[\int_{\mcr}\nabla\tips\nabla\zeta\,d\mcl^2=
\int_{\mcr}\gamma(\tips)\zeta\,d\mcl^2- 2\int_{\tilde v(1)}^{\tilde
v(\varsigma)}(\widetilde Q-2gY)^{1/2}\zeta(L,Y)\,dY.\] Let
$M:=\tilde v(1)$ and $N:=\tilde v(\varsigma)$. Since $\tips$ is even
with respect to the line $X=L$, it follows that
\begin{align}\int_{\mcr_-}\nabla\tips\nabla\zeta\,d\mcl^2=
\int_{\mcr_-}\gamma(\tips)\zeta\,d\mcl^2- \int_{M}^{N}(\widetilde
Q-2gY)^{1/2}\zeta(L,Y)\,dY,\label{grea}\\\text{for all
$\zeta\in\Lip(\mcr_-)$ with $\zeta=0$ on
$(\partial\mcr_-)\setminus\mcr_L$.}\non\end{align}  To show that
this is not possible, we use a blow-up argument. Let
$\{\veps_k\}_{k\geq 1}$ be a sequence with $\veps_k\searrow 0$ as
$k\to\infty$. For any $k\geq 1$, let $\tips^k:\mcr_-\to\bdr$ be
given by
\[\tips^k(X,Y):=\frac{1}{\veps_k}\tips (L+\veps_k(X-L), M+\veps_k(Y-M))
\quad\text{for all }(X,Y)\in\mcr_-.\] Let $\zeta\in\Lip(\mcr_-)$
with $\zeta=0$ on $(\partial\mcr_-)\setminus\mcr_L$. We extend
$\zeta$ to a Lipschitz function in $\{(X,Y):X<L, Y\in\bdr\}$, with
the value $0$ outside of $\mcr_-$. By applying (\ref{grea}) to the
function $\zeta^k:\mcr_-\to\bdr$ given by
\[\zeta^k(X,Y):=\zeta\left(L+\frac{1}{\veps_k}(X-L), M+\frac{1}{\veps_k}(Y-M)\right)
\quad\text{for all }(X,Y)\in\mcr_-,\] we deduce, after a change of
variables in the integrals, that
\begin{align}\int_{\mcr_-}\nabla\tips^k\nabla\zeta\,d\mcl^2=
\int_{\mcr_-}&\veps_k\gamma(\veps_k\tips^k)\zeta\,d\mcl^2\label{yvc}\\&-
\int_{M}^{N}(\widetilde
Q-2gM-2g\veps_k(Y-M))^{1/2}\zeta(L,Y)\,dY.\non\end{align} Since the
family $\{\tips^k\}_{k\geq 1}$ is equi-Lipschitz on $\mcr_-$, there
exists a function $\hat\psi\in\Lip(\mcr_-)$ such that, along a
subsequence (not relabeled),
\begin{align}
&\tips^k\to\hat\psi\quad\text{uniformly on }\mcr_-,\non\\
&\nabla\tips^k\to\nabla\hat\psi\quad\text{weak* in
}L^\infty(\mcr_-).\non\end{align} Since $\tips=0$ on $\mathcal{J}$,
it follows that $\hat\psi=0$ on $\mcr_L\cap\mathcal{J}$. Also, by
passing to the limit as $k\to\infty$ in (\ref{yvc}), we conclude
that \begin{align}\int_{\mcr_-}\nabla\hat\psi\nabla\zeta\,d\mcl^2&=
- \int_{M}^{N}(\widetilde
Q-2gM)^{1/2}\zeta(L,Y)\,dY\label{yvc1}\\&\text{for all
$\zeta\in\Lip(\mcr_-)$ with $\zeta=0$ on
$(\partial\mcr_-)\setminus\mcr_L$}.\non\end{align} This shows in
particular that $\hat\psi$ is a harmonic function in $\mcr_-$. Let $
\mathcal{J}_0:=\{(\tilde u(s), \tilde v(s)):s\in
J\setminus\mathbb{Z}\}$. Since $\hat\psi=0$ on
$\mcr_L\cap\mathcal{J}_0$, the Reflection Principle shows that
$\hat\psi$ can be extended as a harmonic function, odd with respect
to the line $X=L$, in
$\mcr_-\cup(\mcr_L\cap\mathcal{J}_0)\cup\mcr_+$. Let the extension
be denoted also by $\hat\psi$. Then the holomorphic function
$f:=\hat\psi_{_X}-i\hat\psi_{_Y}$ in
$\mcr_-\cup(\mcr_L\cap\mathcal{J}_0)\cup\mcr_+$ satisfies
$f=-(\widetilde Q-2gM)^{1/2}$ on $\mcr_L\cap\mathcal{J}_0$. Since
any holomorphic function on a connected domain is uniquely
determined by its values on any set which has a limit point in that
domain \cite[Theorem 10.18]{Ru}, it follows that
$f(X+iY)=-(\widetilde Q-2gM)^{1/2}$ for all $(X,Y)\in
\mcr_-\cup(\mcr_L\cap\mathcal{J}_0)\cup\mcr_+$. Hence necessarily
$\hat\psi(X,Y)=-(\widetilde Q-2gM)^{1/2}(X-L)$ for all
$(X,Y)\in\mcr_-$. But this contradicts (\ref{yvc1}), since
$\widetilde Q-2gM>0$ and $0<M<N$. This shows that $\varsigma=1$.

Since $\sigma=0$ and $\varsigma=1$, it has been therefore proved
that (\ref{fbi}) holds, $\mathcal{I}$ and $\mathcal{J}$ are empty,
and that $\wids=\{(\tilde u(s),\tilde v(s)):s\in\bdr\}$. Note now
from (\ref{co1}) that, for any compact set
$\mathcal{K}\subset\bdr^2$, \be \text{$\mathcal{K}\subset
\wido\cup\mcb_0$ implies
 $\mathcal{K}\subset\Om_j\cup\mcb$ for all $j$ sufficiently large.}\label{co3}\ee
It is a consequence of (\ref{weil}) that
\[\int_{\wido}\nabla\tips\nabla\zeta\,d\mcl^2=\int_{\wido}
\gamma(\tips)\zeta\,d\mcl^2\label{tiw}\quad\text{for all }\zeta\in
C^1_0(\wido).\] It follows that $\tips\in
C^{3,\alpha}_\loc(\wido\cup\mcb)$ satisfies
\be\Delta\tips=-\gamma(\tips)\quad\text{in }\wido.\label{yte}\ee
Since $\tips\in\Lip(\bdr^2_+)$, Proposition \ref{weha} ensures that
the partial derivatives of $\tips$ have non-tangential limits
$\mch^1$-almost everywhere on $\widetilde\mcs$. Taking into account
(\ref{nond}), we write (\ref{gree}) in the form
\begin{align}\int_{\wido}\nabla\tips\nabla\zeta\,d\mcl^2=
&\int_{\wido}\gamma(\tips)\zeta\,d\mcl^2\non\\&+
\int_{\bdr}\frac{\partial\tips}{\partial n}(\tilde u(s),\tilde
v(s))\zeta(\tilde u(s),\tilde v(s))(\tilde u'(s)^2+\tilde
v'(s)^2)^{1/2}\,ds\non\\&\qquad\qquad\qquad\qquad\qquad\text{ for
all }\zeta\in C^1_0(\bdr^2_+).\label{wewe}\end{align} By comparing
(\ref{weil}) and (\ref{wewe}), we deduce that \be
-\frac{\partial\tips} {\partial n}(\tilde u(s),\tilde v(s))(\tilde
u'(s)^2+\tilde v'(s)^2)^{1/2}=(\widetilde Q-2g\tilde
v(s))^{1/2}\tilde\ell\quad\text{for a.e.\ }s\in\bdr.\label{rela}\ee

Note now that, in view of (\ref{eqil}), there is no loss of
generality in
 assuming that
 \[ \psi^j\to\tips\quad\text{in }C^{0,\alpha}(\bdr^2_+).\]
 This implies that
 \be \gamma(\psi^j)\to\gamma(\tips)\quad\text{in }C^{0,\alpha}(\bdr^2_+).\label{cq1}\ee
  Since (\ref{yte}), (\ref{co3}) and (\ref{cq1}) hold, standard elliptic estimates
\cite[Theorem 4.6 and Theorem 4.11]{GT} show that \be\psi^j\to
\tips\quad\text{in }C^{2,\alpha}_\loc(\wido\cup\mcb).\label{po1}\ee
Now, Theorem \ref{bxv} shows that \be\label{oyt}
\text{$T[\psi^j]\leq 0$ in $\Om_j$ for all $j\geq1$.} \ee We deduce
from (\ref{po1}) and (\ref{oyt}) that
\be\label{tre}\text{$T[\tips]\leq 0$ in $\wido$.} \ee
 It follows from (\ref{tre}) that, in
the notation of (\ref{gra}), \be |\nabla\tips(X,Y)|^2+2gY-\widetilde
Q\leq 0\quad\text{ for $\mch^1$-almost every
$(X,Y)\in\widetilde\mcs$}.\label{qm2}\ee Since $\tips=0$ on $\wids$,
it follows, by using (\ref{qm2}) and Lemma \ref{tg}, upon taking
into account (\ref{nond}), that
 \be0\leq -\frac{\partial\tips}{\partial n}(\tilde u(s),\tilde v(s))\leq
 (\widetilde Q-2g\tilde v(s))^{1/2}\quad\text{ for almost every }s\in\bdr.
 \label{ipe}\ee   It
follows from (\ref{rela}), (\ref{papa}) and (\ref{ipe}) that
\begin{gather}\tilde u'(s)^2+\tilde
v'(s)^2=\tilde\ell^2\quad\text{for almost every
}s\in\bdr,\non\\\frac{\partial\tips}{\partial n}(X,Y)=
 -(\widetilde Q-2gY)^{1/2}\quad\text{ for $\mch^1$-almost every
 $(X,Y)\in\widetilde\mcs$}.\non\end{gather} This completes the proof of the fact
that $(\wids, \mcb_0, \tips, \widetilde Q)$ is a weak solution of
(\ref{apb}).

We now recall for easy reference the following version of the
maximum principle \cite[Lemma 1, p.\ 519]{Eva}, in which we
emphasize that there is no assumption on the sign of the coefficient
$c:\mcg\to\bdr$.

\begin{proposition}\label{maxp} Let $\mcg\subset\bdr^n$, where $n\geq 1$, be a connected
open set. Let $c\in L^{\infty}(\mcg)$ and $w\in C^2(\mcg)$ with
$w\geq 0$ in $\mcg$ be such that
\[\Delta w+cw\leq 0\quad\text{in }\mcg.\]
Then either $w\equiv 0$ in $\mcg$, or $w>0$ in $\mcg$.
\end{proposition}

Since $\psiy^j<0$ in $\Om_j$ for all $j\geq1$, it follows that
$\tips_{_{Y}}\leq 0$ everywhere in $\wido$. Since \be
\Delta\tips_{_{Y}}= -\gamma'(\tips)\tips_{_{Y}}\quad\text{in
}\wido,\ee Proposition \ref{maxp} shows that $\tips_{_{Y}}< 0$ in
$\wido$. As $\tips$ is clearly even in the $X$ variable, it follows
that (\ref{ex2}) holds.

If, in addition, (\ref{e5}) holds, then obviously $\widetilde
Q-2g\tilde v(0)=0$, so that  $(\wids,\mcb_0,\tips,$\\ $\widetilde
Q)$ is an extreme wave. This completes the proof of Theorem
\ref{exi}.
\end{proof}

\begin{proof}[Proof of Theorem \ref{zxc}]
Any solution $(\mcs,\mcb_0,\psi, Q)$ of (\ref{apb}) which belongs to
$\mcc$ has the properties (\ref{ex2}) and (\ref{e1}). It has been
proved in \cite[Theorem 2.3]{EV4}, improving on an earlier result in
\cite{CoS1}, that, if $\gamma(r)\leq 0$ and $\gamma'(r)\geq 0$ for
all $r\in[0,B]$, then any solution $(\mcs,\mcb_0,\psi, Q)$ of
(\ref{apb}) with the properties (\ref{ex2}) and (\ref{e1}) satisfies
\[\max_{\overline\Om}\psiy=\psiy(0,\eta(0)),\] where
$\mcs=\{(X,\eta(X)):X\in\bdr\}$ and $\Om$ is the domain whose
boundary consists of $\mcs$ and $\mcb_0$. Hence (\ref{e5}) follows
from (\ref{vmz}).

The fact that $Q$ is bounded above along $\mcc$ whenever
$\gamma(0)<0$ and $\gamma(r)\leq 0$ for all $r\in[0,B]$ is an
immediate consequence of an estimate in \cite[Proof of Lemma
7.1]{CoS}. I am grateful to Adrian Constantin for pointing out this
to me. Let $f:(2\hg_{\textnormal{max}},\infty)\to\bdr$ be given by
\[ f(\lambda)=\lambda+2g\int_0^B(\lambda-2\hg(r))^{-1/2}\,dr\quad\text{for all }
\la\in(2\hg_{\textnormal{max}},\infty),\] and let
$(\mcs,\mcb_0,\psi, Q)$ be a solution of (\ref{apb}) which belongs
to $\mcc$. It is proven there that $\psiy^2(0,\eta(0))<\la_0$, where
$\la_0$ is the unique solution in $(2\hg_{\textnormal{max}},\infty)$
of the equation $f'(\lambda)=0$.
 It is also proven there that, if
$\psiy^2(0,\eta(0))>2\hg_{\textnormal{max}}$, then \be Q<
f(\psiy^2(0,\eta(0))).\label{beQ}\ee If $\gamma(r)\leq 0$ for all
$r\in[0,B]$, then $\hg_{\textnormal{max}}=0$. Since the restriction
of $f$ to the interval $(0,\la_0)$ is bounded above whenever
$\gamma(0)<0$ and $\gamma(r)\leq 0$ for all $r\in[0,B]$, it follows
from (\ref{beQ}) that $Q$ is bounded above along $\mcc$ in that
case. Therefore (\ref{e6}) holds.

The required conclusion follows from Theorem \ref{exi}. This
completes the proof of Theorem \ref{zxc}.

\end{proof}

\section{On the Stokes conjecture}

In this section we study the shape of the profile of an extreme wave
in a neighbourhood of a stagnation point. In view of Proposition
\ref{tran}, there is no loss of generality in considering only
extreme waves for which $Q=0$. With the origin a stagnation point,
we are interested in the shape of $\mcs$ close to the origin.

Let $(\mcs,\mcb_F, \psi, 0)$ be an extreme wave, where $F<0$, such
that \be\psiy<0\quad\text{in }\Om,\label{kel}\ee and
$\mcs=\{(u(s),v(s)):s\in\bdr\}$, where \bese\label{arfa}
\begin{align}&\text{the mapping $s\mapsto (u(s),v(s))$ is
injective,}\label{i00}\\ &u(0)=v(0)=0, \label{i0}\\ &\text{$s\mapsto
u(s)$ is
 nondecreasing on $\bdr$, \label{ia}}\\&\text{there exist $d,\,e\in\bdr$
 with $d<0<e$ such that $s\mapsto v(s)$ is}\non\\
&\qquad\qquad\qquad\text{nondecreasing on $[d,0]$ and nonincreasing
on $[0,e]$.}\label{ib} \end{align}\ese
 We further assume that
\be\label{rne}\text{$T[\psi]\leq 0$ in $\Om$.} \ee

\begin{remark}{\rm Although
Theorem \ref{bxv} suggests that (\ref{rne}) may be true for all weak
solutions of (\ref{apb}) for which (\ref{kel}) holds, we have so far
not been able to prove this. Note however that, as (\ref{tre})
shows, (\ref{rne}) holds for weak solutions of (\ref{apb}) which
arise as limits of sequences of classical solutions as in Theorem
\ref{exi}.}
\end{remark}

The main result of this section is a proof of the Stokes conjecture
in the following form.

\begin{theorem}\label{sto}Let $(\mcs,\mcb_F,\psi, 0)$ be an extreme
wave which satisfies {\rm(\ref{kel})}-{\rm (\ref{rne})}. In
addition, suppose that $\mcs$ and $\psi$ are symmetric with respect
to the vertical line $X=0$. Then
\[\text{either }\lim_{s\to 0\pm}\frac{v(s)}{u(s)}=\mp\frac{1}{\sqrt 3}
\quad\text{or }\lim_{s\to 0\pm}\frac{v(s)}{u(s)}=0.\] Moreover, if
$\gamma(r)\geq 0$ for all $r\in[0,\delta]$, for some $\delta\in
(0,B]$, then
\[\lim_{s\to 0\pm}\frac{v(s)}{u(s)}=\mp\frac{1}{\sqrt
3}.\]
\end{theorem}

The proof of Theorem \ref{sto} is obtained by combining Theorem
\ref{tblow}, Theorem \ref{uniq} and Proposition \ref{sdf} below, and
will be given after the proofs of those results.

\begin{remark}\label{conj}{\rm We conjecture that the result of Theorem \ref{sto}
continues to hold if the assumption of symmetry of $\mcs$ and $\psi$
is dropped, but this is open even when $\gamma\equiv 0$.}
\end{remark}

\begin{remark}{\rm The existence of trivial extreme waves, noted in Remark \ref{twq}, shows that
the possibility that $\lim_{s\to 0\pm}\frac{v(s)}{u(s)}=0$ in
Theorem \ref{sto} cannot in general be ruled out under the
assumptions there.}
\end{remark}

 We study the asymptotics near the origin of
extreme waves $(\mcs,\mcb_F,\psi, 0)$ satisfying
{\rm(\ref{kel})}-{\rm (\ref{rne})} by means of a blow-up argument
fully described in the proof of Theorem \ref{tblow}. The limiting
problem obtained is the following: find a locally rectifiable curve
$\wids=\{(\tilde u(s),\tilde v(s)):s\in\bdr\}$, where \bese
\label{blow}\begin{align} &\text{$s\mapsto (\tilde u(s),\tilde v(s))$ is injective on $\bdr$},\label{wa}\\
&\tilde u(0)=0,\quad\tilde v(0)=0,\label{wb}\\
&\text{$s\mapsto \tilde u(s)$ is
 nondecreasing on $\bdr$}\label{wc}\\
 &\text{$s\mapsto\tilde v(s)$ is
nondecreasing on $(-\infty,0]$ and nonincreasing on
$[0,\infty)$,}\label{wf}\\&\text{$\lim_{s\to\pm\infty}(|\tilde
u(s)|+|\tilde v(s)|)=\infty$.}\label{we}
\end{align}
such that there exists a function $\tips$ in the unbounded domain
$\wido$ below $\wids$, which satisfies
\begin{align} &\Delta\tips=0\quad\text{in }\wido,\label{ha}\\
&\psi\in\Lip_\loc(\wido\cup\wids),\label{fafa}\\
& \tips\geq 0\quad\text{on }\wido\qquad\text{and}\qquad
\tips_{_Y}\leq 0\quad\text{on }\wido,\label{mg}\\
&\tips=0\quad\text{on }\wids,\label{mh}\\
&|\nabla\tips|^2+2gY=0\quad\text{$\mch^1$-almost everywhere on
$\wids$}\label{aqw}.
\end{align} \ese
Note that, in view of (\ref{ha}) and (\ref{fafa}), the partial
derivatives of $\tips$ have non-tangential limits $\mch^1$-almost
everywhere on $\wids$. The requirement (\ref{aqw}) refers to these
non-tangential boundary values.

\begin{theorem}\label{tblow}Let $(\mcs,\mcb_F,\psi, 0)$ be an extreme wave
 which satisfies {\rm(\ref{kel})}-{\rm (\ref{rne})}. Let
\begin{align}\mathcal{Q}:=\{q\in[-\infty,0]:&\text{ there exists a sequence
$\{\veps_j\}_{j\geq 1}$ with $\veps_j\searrow 0$ as
$j\to\infty$}\non\\&\text{ such that }
\frac{v(\veps_j)}{u(\veps_j)}\to q\text{ as
}j\to\infty\}.\label{liq}\end{align}

If $q\in\mathcal{Q}$ and $q\neq -\infty$, then there exists a
solution $(\wids,\tips)$ of {\rm (\ref{blow})} with $\tilde v(\tilde
s)=q\tilde u(\tilde s)$ for some $\tilde s\in(0,\infty)$.

 If $-\infty\in\mathcal{Q}$, then there exists a solution $(\wids,\tips)$ of {\rm (\ref{blow})}
with $\tilde u(\tilde s)=0$ for some $\tilde s\in(0,\infty)$.

Moreover, if $\mcs$ is symmetric with respect to the line $X=0$,
then $-\infty\notin\mathcal{Q}$.
\end{theorem}

Note that problem (\ref{blow}) has a trivial solution
$(\wids_0,\tips_0)$ where $\wids_0=\{(X,0):X\in\bdr\}$ and $\tips_0=
0$ in $\bdr^2_-$, the lower half-plane. Any other solution of
(\ref{blow}) is called a \emph{nontrivial solution}.

There also exists an explicit nontrivial solution of (\ref{blow}),
known as the \emph{Stokes corner flow}.
 Let $\wids^*:=\{(X,\eta^*(X)):X\in\bdr\}$, where
\be\eta^*(X):=-{\ds\frac{1}{\sqrt 3}}|X|\quad\text{for all
}X\in\bdr.\label{etas}\ee Let $\wido^*$ be the domain below
$\wids^*$, and let the harmonic function $\tips^*$ in $\wido^*$ be
given, for all $(X,Y)\in\wido^*$, by
\be\tips^*(X,Y):=\frac{2}{3}g^{1/2}\Im \Big(i(iZ)^{3/2}\Big)
\quad\text{where }Z=X+iY.\label{fisa}\ee Then $(\wids^*,\tips^*)$ is
a nontrivial solution of (\ref{blow}).

\begin{theorem}\label{uniq} The only nontrivial solution
$(\wids,\tips)$ of {\rm(\ref{blow})} for which both $\wids$ and
$\tips$ are symmetric with respect to the vertical line $X=0$ is the
Stokes corner flow $(\wids^*,\tips^*)$.
\end{theorem}

\begin{remark} {\rm We conjecture that the result of Theorem \ref{uniq}
continues to hold if the assumption of symmetry of $\wids$ and
$\tips$ is dropped. If this were the case, the validity of the
conjecture in Remark \ref{conj} would immediately follow. It is
conceivable that the moving-planes method could be used to prove the
symmetry of all solutions of (\ref{blow}). This method has so far
been successfully used to prove the symmetry of various types of
hydrodynamic waves, see \cite{CEW} for references. The main
difficulty in the present situation is the lack of any estimates on
the behavior of $(\wids,\tips)$ at infinity. If good enough
estimates of this type were available, the desired result would
follow, see \cite{CS} for a related situation and \cite[Theorem
3.1]{EV} for how to deal with the presence of a stagnation point.}
\end{remark}

The following simple result, which will be used in the proofs of
Theorem \ref{sto} and Theorem \ref{asco} below, is also of some
interest in itself.

\begin{proposition}\label{sdf} Suppose that $\gamma(r)\geq 0$ for all $r\in[0,\delta]$,
for some $\delta\in (0,B]$. Let $(\mcs,\mcb_F,\psi, 0)$ be an
extreme wave,
 where {\rm(\ref{kel})},
{\rm (\ref{i00})-(\ref{ia})} hold, and $T[\psi]\leq 0$ in $\Om$.
Then $\Om$ does not contain any truncated cone with vertex at the
origin and opening angle greater that $120^\circ$.
\end{proposition}

 The next result is new even for irrotational waves, in
that the symmetry of $\mcs$ and $\psi$ is not required. The drawback
is that the existence of lateral tangents at the stagnation point is
an assumption.

\begin{theorem}\label{asco} Let $(\mcs,\mcb_F,\psi, 0)$ be an extreme wave
 which satisfies {\rm(\ref{kel})}-{\rm (\ref{rne})}. Suppose that
there exist $q_\pm\in[0,\infty]$ such that $\lim_{s\to
0\pm}\frac{|v(s)|}{|u(s)|}=q_\pm$. Then either $q_\pm=\frac{1}{\sqrt
3}$ or $q_\pm=0$. Moreover, if $\gamma(r)\geq 0$ for all
$r\in[0,\delta]$, for some $\delta\in(0,B]$,
 then
$q_\pm=\frac{1}{\sqrt 3}$.
\end{theorem}

We now give the proofs of the results of this section.

\begin{proof}[Proof of Theorem \ref{tblow}] There is clearly no
loss of generality in assuming that the properties (\ref{arfa}) of
are satisfied by a parametrization of $\mcs$ by arclength, i.e.,
 $\mcs=\{(u(s),v(s)):s\in\bdr\}$,  where $u,\,v\in\Lip(\bdr)$
 satisfy  \[
 u'(s)^2+v'(s)^2=1\quad\text{for almost every }s\in\bdr.\]
  We extend $\psi$ to $\bdr^2$ with the value $0$ on the connected
component of $\bdr^2\setminus\overline\Om$ whose boundary is $\mcs$,
and with the value $B$ on the component whose boundary is $\mcb$.
The extension, denoted also by $\psi$, is a Lipschitz function on
$\bdr^2$. It is an immediate consequence of the assumption
(\ref{rne}) that there exists a constant $K>0$ such that
\be|\nabla\psi(X,Y)|^2\leq K|Y|\quad\text{for $\mcl^2$-almost every
$(X,Y)\in\bdr^2$}.\label{lipq}\ee

Let $q\in\mathcal{Q}$ and let the sequence $\{\veps_j\}_{j\geq 1}$
with $\veps_j\searrow 0$ as $j\to\infty$ be such that
${v(\veps_j)}/{u(\veps_j)}\to q$ as $j\to\infty$. Let us consider
the following sequence of rescalings of the domain $\Om$ and the
function $\psi$. For any $j\geq 1$, let
\be\Om_j:=\frac{1}{\veps_j}\Om,\ee
 and $\psi^j:\bdr^2\to\bdr$ be given by
\be\psi^j(X,Y):=\frac{1}{\veps^{3/2}_j}\psi(\veps_j X,\veps_j
Y)\quad\text{for all $(X,Y)\in\bdr^2$}.\label{xdey}\ee The boundary
of the domain $\Om_j$ consists of the curve
$\mcs_j:={\veps_j^{-1}}\mcs$ and the horizontal line
$\mcb_{F/{\veps_j}}$.
 The curve $\mcs_j$ is ${2L}{\veps_j^{-1}}\,$-periodic in the horizontal
 direction, and
 can be parametrized by arclength by means of the
functions $u_j,\, v_j:\bdr\to\bdr$ given by
\[u_j(s)=\frac{1}{\veps_j}u(\veps_j s),\quad v_j(s)=\frac{1}{\veps_j}v(\veps_j s)
\quad\text{for all }s\in\bdr .\] The function $\psi^j$ is also
${2L}{\veps_j^{-1}}\,$-periodic in the horizontal
 direction and is a weak solution of
 \begin{subequations}\label{xapb}
\begin{align}
& \Delta\psi^j=-\veps_j^{1/2}\gamma(\veps_j^{3/2}\psi^j)\quad\text{in }\Om_j,\label{xap0}\\
& \psi^j= B\veps_j^{-3/2}\quad \text{on
}\mcb_{F/\veps_j},\label{xap1}\\&\psi^j=0\quad \text{on }\mcs_j,
\label{xap2}\\&\vert\nabla\psi^j\vert^{2}+2gY=0 \quad\text{on }
\mcs_j\label{xap3}.
 \end{align}
\end{subequations}
In particular, for any $\zeta\in C_0^1(\bdr^2)$, the following holds
for all $j$ sufficiently large:
\begin{align}\int_{\Om_j}\nabla\psi^j\nabla\zeta\,d\mcl^2=\int_{\Om_j}&\veps_j^{1/2}
\gamma(\veps_j^{3/2}\psi^j)\zeta\,d\mcl^2\non\\ &- \int_\bdr(-2g
v_j(s))^{1/2}\zeta(u_j(s), v_j(s))\,ds.\label{ccila}\end{align} It
is immediate from (\ref{lipq}) and (\ref{xdey}) that \be \text{the
family $\{\psi^j\}_{j\geq 1}$ is equi-Lipschitz in any horizontal
strip $\mcg\subset\bdr^2$.}\label{65}\ee It follows that there exist
functions $\tips\in\Lip_\loc(\bdr^2)$ and $\tilde u,\,\tilde
v\in\Lip(\bdr)$ such that, along a subsequence (not relabeled),
\begin{align}
&\psi^j\to\tips\quad\text{uniformly on any compact set $\mck\subset\bdr^2$},\label{h2}\\
&\nabla\psi^j\to\nabla\tips\quad\text{weak* in $L^\infty(\mcg)$
 for any horizontal strip $\mcg\subset\bdr^2$},\label{h3}\\
&u_j\to\tilde u,\, v_j\to\tilde v \quad\text{uniformly on any
compact
subset of }\bdr,\label{h4}\\
&u_j'\to\tilde u',\, v_j'\to\tilde v'\quad\text{weak* in
}L^\infty(\bdr).\label{h5}
\end{align}
It is immediate that \be\tilde u'(s)^2+\tilde v'(s)^2\leq
1\quad\text{for almost every }s\in\bdr.\label{oaie}\ee By arguing as
in the proof of (\ref{taa}), we deduce that, for every
$a,\,b\in\bdr$ having the same sign, \be |b-a|\leq|\tilde
u(b)-\tilde u(a)|+|\tilde v(b)-\tilde v(a)|.\label{maa}\ee
Therefore,
\begin{gather}
 \text{the
mapping $s\mapsto(\tilde u(s),\tilde v(s))$ is
injective on $(-\infty,0]$ and on $[0,\infty)$,}\label{lija}\\
1\leq|\tilde u'(s)|+|\tilde v'(s)|\quad\text{for almost every
}s\in\bdr.\label{nonda}
\end{gather}
It is obvious that (\ref{wb})-(\ref{we}) hold.

We would now like to prove that (\ref{wa}) holds. Let $\sigma\in
[0,\infty]$ be such that
 \be \sigma:=\sup\{s\in[0,\infty):\tilde u(\pm s)=0\}.\label{alba}\ee
 To prove (\ref{wa}) it suffices, in view of (\ref{wc}), (\ref{wf})
 and (\ref{lija}), to show that $\sigma=0$. Note from (\ref{wc}), (\ref{wf}), (\ref{oaie}) and
(\ref{nonda}) that \bese\label{sfg}\begin{gather} \tilde u
'(s)=0\text{ for a.e. }s\in(-\sigma,\sigma)\\ \tilde v'(s)=1\text{
for a.e. }s\in(-\sigma,0),\qquad \tilde v'(s)=-1\text{ for a.e.
}s\in(0,\sigma).\end{gather}\ese

We now claim that $\sigma\in[0,\infty)$. Suppose for a contradiction
that $\sigma=+\infty$. It is immediate from (\ref{h4}) and
(\ref{sfg}) that, for any compact set $\mathcal{K}\subset\bdr^2$,
\be \text{ $\mathcal{K}\subset\bdr^2\setminus \{(0,Y): Y\leq0\}$
implies $\mathcal{K}\subset\mathcal{V}_j$ for all $j$ sufficiently
large,} \label{cola}\ee where, for any $j\geq 1$, $\mathcal{V}_j$ is
the component of $\bdr^2\setminus(\Om_j\cup\mcs_j)$ whose boundary
is $\mcs_j$. It follows from (\ref{cola}) that $\tips=0$ in
$\bdr^2\setminus \{(0,Y): Y\leq0\}$ and hence, using the continuity
of $\tips$ in $\bdr^2$, that $\tips=0$ in $\bdr^2$. Moreover, by
passing to the limit as $j\to\infty$ in (\ref{ccila}), we obtain,
taking also into account (\ref{sfg}), that
\[\int_{-\infty}^0(-2gY)^{1/2}\zeta (0,Y)\,dY=0\quad\text{for all }
\zeta\in C_0^1(\bdr^2).\] Since this is clearly not possible, it
follows that $\sigma\in[0,\infty)$.

  Let
\[\mathcal{I}:=\{(\tilde u(s),\tilde v(s)):s\in
(-\sigma,\sigma)\},\qquad\wids:=\{(\tilde u(s),\tilde
v(s)):s\in\bdr\setminus (-\sigma,\sigma)\}.
\]
 Then $\mathcal{I}$ is either empty or a half-open vertical segment,
 while $\wids$ is a locally rectifiable curve. Let $\wido$ be the unbounded domain below $\wids$.
  We first show
  that $\mathcal{I}$ is empty, and then that $(\wids,\tips)$ is a solution
 of (\ref{blow}).

 It is
immediate from (\ref{h4}) that, for any compact set
$\mathcal{K}\subset\bdr^2$,
\begin{gather}\text{$\mathcal{K}\subset\wido$ implies $\mathcal{K}\subset \Om_j$ for all $j$
sufficiently large,}\label{py}\\
\text{ $\mathcal{K}\subset\bdr^2\setminus(\wido\cup\wids\cup
\mathcal{I})$ implies $\mathcal{K}\subset\mathcal{V}_j$ for all $j$
sufficiently large.}\label{qy}
\end{gather}
 It is obvious that $\tips\geq 0$ in $\wido$.
Also, it follows from (\ref{qy}) that $\tips=0$ in
$\bdr^2\setminus(\wido\cup\wids\cup \mathcal{I})$, and hence, using
the continuity of $\tips$ in $\bdr^2$, that $\tips=0$ on
$\bdr^2\setminus\wido$. The validity of (\ref{h2})-(\ref{h4}) makes
it possible to pass to the limit as $j\to\infty$ in (\ref{ccila}),
to obtain \be\int_{\wido}\nabla\tips\nabla\zeta\,d\mcl^2= -
\int_\bdr(-2g\tilde v(s))^{1/2}\zeta(\tilde u(s),\tilde
v(s))\,ds\quad\text{for all }\zeta\in C^1_0(\bdr^2).\label{weila}\ee

We now claim that $\sigma=0$. Suppose for a contradiction that this
is not so. It is a consequence of (\ref{sfg}) that $\tilde
v(-\sigma)=\tilde v(\sigma)$. Let $\mathcal{\mcd}$ be the disc
centred at $(0,0)$ and with the point $(0, \tilde v(\sigma))$ on its
boundary. It follows from (\ref{weila}) and (\ref{sfg}) that
\[\int_{\tilde v(\sigma)}^{0}(-2gY)^{1/2}\zeta(0,Y)\,dY=0\quad\text{for all }\zeta\in
C^1_0(\mathcal{\mcd}).\] Since this is clearly not possible, it
follows that $\sigma=0$.

It has been therefore proved that $\mathcal{I}$ is empty,
$\wids=\{(\tilde u(s),\tilde v(s)):s\in\bdr\}$, and that
(\ref{wa})-(\ref{we}) hold. It is a consequence of (\ref{weila})
that
\[\int_{\wido}\nabla\tips\nabla\zeta\,d\mcl^2=0\label{tiwa}\quad\text{for
all }\zeta\in C^1_0(\wido).\] It follows that $\tips\in
C^\infty(\wido)$ satisfies \be\Delta\tips=0\quad\text{in
}\wido.\label{bg}\ee The condition $\tips\in\Lip_\loc(\bdr^2)$
ensures that the partial derivatives of $\tips$ have non-tangential
limits $\mch^1$-almost everywhere on $\wids$. It follows from
(\ref{gree}), upon taking into account (\ref{nonda}), that
\begin{align}\int_{\wido}\nabla\tips\nabla\zeta\,d\mcl^2=
 \int_{\bdr}\frac{\partial\tips}{\partial n}(\tilde u(s),\tilde
v(s))\zeta(\tilde u(s),\tilde v(s))(\tilde u'(s)^2+\tilde
v'(s)^2)^{1/2}\,ds\non\\\text{ for all }\zeta\in
C^1_0(\bdr^2).\label{wewea}\end{align} By comparing (\ref{weila})
and (\ref{wewea}),
 we deduce that
 \be -\frac{\partial\tips}{\partial n}(\tilde u(s),\tilde
v(s))(\tilde u'(s)^2+\tilde v'(s)^2)^{1/2}=(-2g\tilde
v(s))^{1/2}\quad\text{for a.e.\ }s\in\bdr.\label{relaa}\ee

It is a consequence of (\ref{65}) that
 \be \veps_j^{1/2}\gamma(\veps_j^{3/2}\psi^j)\to 0\quad\text{in }C^{0,\alpha}(\mck)
 \quad\text{for any compact set $\mck\subset\bdr^2$}.\label{uew}\ee
 Since (\ref{xap0}), (\ref{bg}), (\ref{py}) and (\ref{uew}) hold, a standard elliptic estimate
\cite[Theorem 4.6]{GT} shows that \be\psi^j\to \tips\quad\text{in
}C^{2,\alpha}_\loc(\wido).\label{paw}\ee
 Note now that
(\ref{rne}) yields, for all $j\geq 1$,
\begin{align}
|\nabla\psi^j(X,Y)|^2+2gY+\frac{2}{\veps_j}\hg(\veps_j^{3/2}\psi^j(X,Y))-2\veps_j^{1/2}\varpi\psi^j(X,Y)\leq
0\non\\\quad\text{for all }(X,Y)\in\Om_j.\label{maw}\end{align}
 We deduce from (\ref{paw}) and (\ref{maw})
   that \be
|\nabla\tips(X,Y)|^2+2gY\leq 0\quad\text{for all
}(X,Y)\in\wido.\label{121}\ee Since $\tips=0$ on $\wids$, it
follows, by using (\ref{121}) and Proposition \ref{tg}, upon taking
into account (\ref{nonda}), that \be 0\leq
-\frac{\partial\tips}{\partial n}(\tilde u(s),\tilde v(s))\leq
 (-2g\tilde v(s))^{1/2}\quad\text{for almost every $s\in\bdr$}.\label{ipea}\ee
 It follows from  (\ref{relaa}), (\ref{oaie}) and (\ref{ipea})
 that
\begin{gather}\tilde u'(s)^2+\tilde
v'(s)^2=1\quad\text{for almost every
}s\in\bdr,\\\frac{\partial\tips}{\partial n}(X,Y)=
 -(-2gY)^{1/2}\quad\text{ for $\mch^1$-almost every $(X,Y)\in\widetilde\mcs$}.\end{gather}
This completes the proof of the fact that $(\wids,\tips)$ is a
solution of (\ref{blow}).

 If $q\neq -\infty$, then obviously  $\tilde v(1)=q\tilde u(1)$, while if $q=-\infty$, then
 $\tilde u(1)=0$.

 If $\mcs$ is symmetric, the fact that
 $-\infty\not\in\mathcal{Q}$ is an immediate consequence of the fact that $\sigma=0$.
  This completes the proof of Theorem \ref{tblow}.
\end{proof}

\begin{proof}[Proof of Theorem \ref{uniq}] We first show how solutions of
(\ref{blow}) can be described by solutions of the nonlinear integral
equation (\ref{main}).  The required result is then obtained  by
invoking a uniqueness result from \cite{EV} for the integral
equation. In the process of deriving (\ref{main}) we also give a
theory of (not necessarily symmetric) solutions of (\ref{blow}),
concerning the reduction of this free-boundary problem to a problem
in a fixed domain and on the local regularity of solutions. Whilst
problem (\ref{blow}) appears not to have been studied before, there
are obvious similarities to problem (\ref{apb}) for irrotational
waves of finite or infinite depth, treatments of Hardy-space
solutions of which have been given in \cite{ST, EV, EV2, EV3}. To
avoid inessential technicalities, proofs of results for (\ref{blow})
are sometimes not given in situations where they would be obtainable
by routine modifications from proofs in \cite{ST, EV, EV2, EV3}.

Let $(\wids,\tips)$ be any nontrivial solution of (\ref{blow}). It
follows that (\ref{mg}) holds in the form \be\tips > 0\quad\text{on
}\wido\qquad\text{and}\qquad \tips_{_Y}< 0\quad\text{on
}\wido.\label{utq}\ee Since the non-tangential boundary values of
any bounded holomorphic function in a bounded open set whose
boundary is a rectifiable Jordan curve cannot vanish on a set of
positive $\mch^1$ measure unless the function is identically $0$, it
follows from (\ref{wf}), (\ref{aqw}) and (\ref{utq}) that \be \tilde
v(s)<0 \quad\text{for all }s\in\bdr\setminus\{0\}.\label{lk1}\ee
 Let us denote
\be\text{$\wids_+:=\{(\tilde u(s),\tilde v(s)):s\in(0,\infty)\}$ and
$\wids_-:=\{(\tilde u(s),\tilde v(s)):s\in(-\infty, 0)\}$.}\non\ee
Let $W_0:\bdc_+\to\wido$ be a conformal mapping from the upper
half-plane $\bdc_+$ onto $\wido$. By Caratheodory's Theorem, $W_0$
has an extension as a homeomorphism between the closures in the
extended complex plane of these domains. It is also classical that
$W_0$ can be chosen such that it maps the origin onto itself, the
positive real axis onto $\wids_-$ and the negative real axis onto
$\wids_+$. Then $\tips\circ W_0$ is a positive harmonic function in
the upper half-plane and continuous on its closure, with $\tips\circ
W_0=0$ on the real line. Hence there exists $c>0$ such that
\[(\tips\circ W_0)(z)=c y\quad\text{ for all }z=x+iy\in \bdc_+.\]
 Let $\tilde\vfi$ be a harmonic conjugate of
$-\tips$ in $\wido$, so that the function
$\tilde\om:=\tilde\vfi+i\tips$ is holomorphic in $\wido$ and
satisfies
\[(\tilde\omega\circ W_0)(z)=z\quad\text{for all }z\in\bdc_+.\]
Let $W:\bdc_+\to\wido$ be given by $W(z):=W_0(c^{-1}z)$ for all
$z\in \bdc_+$. Then $W$ has the same conformal mapping properties as
$W_0$, and $\tilde\omega$ is the inverse conformal mapping of $W$.
Let us write, for all $x+iy\in\bdc_+$,
\begin{align}W(x+iy)&=U(x,y)+i V(x,y),\label{uv}\\
 W'(x+iy)&=-\exp(\tau(x,y)+i\te(x,y))\label{tt}\end{align} where
$\tau$ and $\te$ are harmonic functions on $\bdr^2_+$. It follows
from (\ref{utq}) that \be -\frac{\pi}{2}< \te(x,y)<
\frac{\pi}{2}\quad\text{for all }(x,y)\in\bdr^2_+.\ee The M.\ Riesz
Theorem now implies that $\tau\in h^p_\bdc(\bdr^2_+)$ for all
$p\in(1,\infty)$. Therefore $\tau$ and $\te$ have non-tangential
boundary values almost everywhere on the real line, from which they
can be recovered by Poisson Formula and which are related to one
another by the Hilbert transform.

For any $x_0\in (0,\infty)$, let $X_0+iY_0:=W(x_0+i0)$, so that
$Z_0:=X_0+iY_0$ is located on $\wids_-$. Let $Z_1$ and $Z_2$ be
located on $\mcs_-$ such that $Z_0$ is situated between $Z_1$ and
$Z_2$ and  that there exist non-tangential limits of $\nabla\psi$ at
$Z_1$ and $Z_2$. Let $\mcg$ be a subdomain of $\wido$ such that the
boundary of $\mcg$ is a rectifiable Jordan curve
$\mathcal{J}:=\mathcal{I}\cup\mathcal{L}$, where $\mathcal{I}$ is
the arc of $\wids$ joining $Z_1$ and $Z_2$, and $\mathcal{L}$ is an
arc contained in $\wido$, joining $Z_1$ and $Z_2$ and which
approaches $\wids$ non-tangentially at $Z_1$ and $Z_2$. By
(\ref{aqw}) and the construction of $\mcg$, the non-tangential
boundary values of the harmonic function $\tau\circ W$ in
$h^p(\mcg)$ are essentially bounded, and therefore $\tau\circ W$ is
bounded in $\mcg$. It follows that there exists a rectangle
$\Pi:=(x_0-\epsilon, x-0+\epsilon)\times(0,\delta)$ in $\bdr^2_+$,
where $0<\epsilon<x_0$ and $\delta>0$, in which $\tau$ is bounded.
This shows that the partial derivatives of $U,\,V$ in (\ref{uv}) are
bounded in $\Pi$, and therefore have non-tangential limits almost
everywhere on $(x_0-\epsilon, x_0+\epsilon)\times\{0\}$. Since
$x_0\in(0,\infty)$ was arbitrary, it follows that the partial
derivatives of $U,\,V$ have non-tangential limits almost everywhere
on the positive real axis. A similar statement can be made for the
negative real axis.

 By arguing as in \cite[Lemma 4.2]{EV3}, we deduce  that the mapping
$ t\mapsto W(t+i0)$ is locally absolutely continuous on each of the
intervals $(0,\infty)$ and $(-\infty, 0)$, and \be\frac{d}{dt}W(t+
i0)=\lim_{(x,y)\to (t,0)} W'(x+iy)\quad\text{for almost every
}t\in\bdr,\label{aco}\ee the above limit being taken
non-tangentially within $\bdr^2_+$. But since the mappings $t\mapsto
U(t,0)$, $t\mapsto V(t,0)$ are monotone on $[0,\infty)$ and on
$(-\infty,0]$, it follows that $t\mapsto W(t+i0)$ is locally
absolutely continuous on $\bdr$.

For any harmonic function $\xi$ in $\bdr^2_+$ which has
non-tangential limits almost everywhere on the real axis, we use
from now on the notation $t\mapsto \xi(t)$ instead of either
$t\mapsto \xi(t,0)$ or $t\mapsto \xi(t+i0)$ to denote the boundary
values of $\xi$.

We deduce from  the free boundary condition (\ref{aqw}) that \be
|W'(t)|^2(-2gV(t))=1\quad\text{for almost every
}t\in\bdr,\label{vara}\ee and therefore \be
\tau(t)=-\log\{(-2gV(t))^{1/2}\}\quad\text{for almost every
}t\in\bdr.\label{taua}\ee It is also obvious that, for almost every
$t\in\bdr$, \be -U'(t)=\frac{\cos\te(t)}{(-2gV(t))^{1/2}},\quad
-V'(t)=\frac{\sin\te(t)}{(-2gV(t))^{1/2}}.\label{cela}\ee (Note that
by (\ref{aco}) the notation $U'(t),\, V'(t)$, for almost every
$t\in\bdr$, is unambiguous.) It follows that $\te(t)$ gives the
angle between the tangent to the curve $\wids$ at the point
$(U(t),V(t))$ and the horizontal, for almost every $t\in\bdr$. Note
also that a consequence of the fact that $\tau\in
h^p_\bdc(\bdr^2_+)$ for all $p\in(1,\infty)$ is that \be\int_\bdr
|\tau(w)|^p\frac{1}{1+w^2}\,dw<+\infty\quad\text{for all
}p\in(1,\infty).\label{inte}\ee

By a bootstrap argument as in \cite[Theorem 3.5]{EV4}, see also
\cite[Theorem 2.3]{EV}, we deduce that $W,\tau,\theta\in
C^\infty(\overline{\bdr^2_+}-\{(0,0)\})$, which implies that
$\wids_+$ and $\wids_-$ are $C^\infty$ curves and $\tips\in
C^\infty(\wido\cup\wids_+\cup\wids_-)$. A classical result of Lewy
\cite{Le} then shows that $\wids_+$ and $\wids_-$ are real-analytic
curves, and $\tips$ has a harmonic extension across  $\wids_+$ and
$\wids_-$.

Integrating the second relation in (\ref{cela}) written in the form
\[-V'(t){(-2gV(t))^{1/2}}={\sin\te(t)}\quad\text{for almost every }t\in\bdr,\]
we obtain, since $V(0)=0$, that
\be(-2gV(y))^{1/2}=\frac{1}{(3g)^{1/3}}\left(\int_0^y\sin\te(w)\,dw\right)^{1/3}\quad\text{for
all }y\in\bdr.\label{iam}\ee The geometric properties of $\wids$
expressed by (\ref{wc}) and (\ref{wf}) imply that \be 0\leq\te\leq
{\pi}/{2}\text{ on }(0,\infty) \quad\text{ and
}\quad-{\pi}/{2}\leq\te\leq 0\text{ on }(-\infty,0).\label{teb}\ee
Moreover, note from (\ref{lk1}) that $V(y)< 0$ for all $y\neq 0$.
This means, in view of (\ref{iam}), that
\be\int_0^y\sin\te(w)\,dw>0\quad\text{for all }y\neq
0.\label{klm}\ee

Suppose now that $\wids$ and $\tips$ are symmetric with respect to
the line $X=0$. It follows that $\tau$ is an even function and $\te$
is an odd function on $\bdr$. The definition of a Hilbert transform
then shows that \be \te(x)=\frac{1}{\pi}\int_0^\infty
\left(\frac{1}{x-y}+\frac{1}{x+y}\right)\{\tau(y)-\tau(x)\}\,dy\quad\text{for
all }x\in(0,\infty).\ee Note from (\ref{taua}) and (\ref{iam}) that
$\tau(y)/y\to 0$ as $y\to\infty$. Using this fact, (\ref{inte}) and
the monotonicity of $\tau$ on $(0,\infty)$, an integration by parts
(the validity of which can be justified as in \cite[Proof of
Proposition 4.3]{EV}) shows that \[
\te(x)=\frac{1}{3\pi}\int_0^\infty
\log\left|\frac{x+y}{x-y}\right|\{-\tau'(y)\}\,dy\quad\text{for all
}x\in(0,\infty).\] This means, upon using (\ref{taua}) and
(\ref{iam}), that \be\te(x)=\frac{1}{3\pi}\int_0^\infty
\log\left|\frac{x+y}{x-y}\right|\frac{\sin\te(y)}{\int_0^y\sin\te(w)\,dw}\,dy\quad\text{for
all }x\in(0,\infty).\label{main}\ee

The following result, which is \cite[Theorem 4.5]{EV}, is the key to
the proof of Theorem \ref{uniq} .
\begin{theorem}\label{uint} The
only solution $\te:(0,\infty)\to\bdr$ of {\rm (\ref{main})} with
$0\leq \te\leq \pi/2$ on $(0,\infty)$ and such that \be
0<\inf_{x\in(0,\infty)} \te(x)\label{asda}\ee is the function
$\te^*:(0,\infty)\to\bdr$ given by $\te^*(x)=\pi/6$ for all
$x\in(0,\infty)$.
\end{theorem}

The following new result shows that (\ref{asda}) is in fact not a
restriction in Theorem \ref{uint}.

\begin{proposition}\label{vsq} Let $\te:(0,\infty)\to\bdr$ be any solution
of {\rm (\ref{main})} with $0\leq \te\leq \pi/2$ on $(0,\infty)$ and
such that \be\int_0^y\sin\te(w)\,dw>0\quad\text{for all
}y\in(0,\infty).\label{qua}\ee Then $\te$ satisfies {\rm
(\ref{asda})}.

\end{proposition}

\begin{proof}[Proof of Proposition \ref{vsq}]
 It is obvious that
\[\te(x)\geq\frac{1}{3\pi}\int_0^x
\log\left|\frac{x+y}{x-y}\right|\frac{1}{y}\,{\sin\te(y)}\,dy\quad\text{for
all }x\in(0,\infty).\] Since for every $x,y\in(0,\infty)$ with
$0<y<x$, the following inequality holds:
\[\log\left|\frac{x+y}{x-y}\right|\geq 2\frac{y}{x},\]
it follows that
\[\te(x)\geq\frac{2}{3\pi}\frac{1}{x}\int_0^x
\,{\sin\te(y)}\,dy\quad\text{for all }x\in(0,\infty).\] From this it
is immediate that
\[\sin\te(y)\geq\frac{4}{3\pi^2}\frac{1}{y}\int_0^y
\,{\sin\te(w)}\,dw\quad\text{for all }y\in(0,\infty).\] We now
deduce from (\ref{main}) that
\[\te(x)\geq\frac{4}{9\pi^3}\int_0^\infty
\log\left|\frac{x+y}{x-y}\right|\,\frac{1}{y}\,dy=\frac{2}{9\pi}\quad\text{for
all }x\in(0,\infty),\] which proves (\ref{asda}).
\end{proof}

Since (\ref{klm}) and (\ref{teb}) hold, it follows from Theorem
\ref{uint} and Proposition \ref{vsq} that, for any symmetric
nontrivial solution $(\wids,\tips)$ of (\ref{blow}), the function
$\te$ associated to it necessarily coincides with $\te^*$, the
constant function $\pi/6$. It is then straightforward that
$(\wids,\tips)$ coincides with $(\wids^*,\tips^*)$ given by
(\ref{etas})-(\ref{fisa}). This completes the proof of Theorem
\ref{uniq}.
\end{proof}

\begin{proof}[Proof of Proposition \ref{sdf}] We use the following particular case of
a result of Oddson \cite{Od}.

\begin{proposition}\label{odo} Let $r_0>0$ and $\mu>1$. Let \[\mcg:=\{re^{it}:0<r<r_0,
|t|<{\pi}/{(2\mu)}\}.\] Let $w\in C^2(\mcg)\cap C(\overline\mcg)$ be
a superharmonic function in $\mcg$, such that $w(0,0)=0$ and $w>0$
in $\overline\mcg\setminus\{(0,0)\}$. Then there exists $\kappa>0$
such that
\[w (re^{it})\geq \kappa r^\mu\cos\mu t\quad\text{in }\overline\mcg.\]
\end{proposition}
 Suppose for a
contradiction that $\Om$ contains such a truncated cone.
 Then there exist
$r_0>0$ and $\alpha_1$, $\alpha_2$ with $-\pi\leq
\alpha_1<\alpha_2\leq 0$ and $\alpha_2-\alpha_1>2\pi/3$, such that
$\overline\mcg\setminus\{(0,0)\}\subset\Om_0$, where
$\mcg:=\{re^{it}:0<r<r_0, \alpha_1<t<\alpha_2\}$ and
$\Om_0:=\{(X,Y)\in\Om:0<\psi(X,Y)<\delta\}$. Since $\psi$ is
superharmonic in $\mcg$,
 $\psi(0,0)=0$ and $\psi>0$ in $\overline\mcg\setminus\{(0,0)\}$,
 Proposition \ref{odo}
shows that there exists $\kappa_0>0$ such that
\[\psi(0,Y)\geq \kappa_0|Y|^\mu\quad\text{for all $Y\in (-r_0,0)$},\]
where $\mu:=\pi/(\alpha_2-\alpha_1)$, so that $\mu<3/2$. But this
contradicts the estimate, see (\ref{lipq}),
\[|\nabla\psi(0,Y)|^2\leq K|Y|\quad\text{for all $Y$ such that $(0,Y)\in\Om$,}\]
which is a consequence of the assumption $T[\psi]\leq 0$. This
completes the proof of Proposition \ref{sdf}.
\end{proof}

\begin{proof}[Proof of Theorem \ref{sto}]
Let $\mathcal{Q}$ be given by (\ref{liq}).
 Obviously, $\mathcal{Q}$ is a closed subinterval of $[-\infty, 0]$. Since $\mcs$ and $\psi$ are symmetric,
it is immediate from Theorem \ref{tblow} and Theorem \ref{uniq} that
$\mathcal{Q}$ is a subset of $\{0,-1/\sqrt{3}\}$. Hence either
$\mathcal{Q}=\{0\}$ or $\mathcal{Q}=\{-1/\sqrt{3}\}$. When
$\gamma(r)\geq 0$ for all $r\in[0,\delta]$, the possibility that
$\mathcal{Q}=\{0\}$ is ruled out by Proposition \ref{sdf}. This
completes the proof of Theorem \ref{sto}.
\end{proof}

\begin{proof}[Proof of Theorem \ref{asco}] Suppose first that $q_+\neq\infty$
and $q_-\neq\infty$.
 Let $(\wids,\tips)$
be the solution of (\ref{blow}) whose existence is given by Theorem
\ref{tblow}. Moreover, the proof of Theorem \ref{tblow} shows that
necessarily $\wids=\{(X,\tilde\eta(X)):X\in\bdr\}$, where
\[\tilde\eta(X):=\left\{\begin{aligned}&-q_+|X|\quad\text{for all }X\in
[0,\infty),
\\&-q_-|X|\quad\text{for all }X\in(-\infty, 0].\end{aligned}\right.\]
We now ask for what values of $q_\pm$ there exist solutions $\tips$
of (\ref{ha})-(\ref{aqw}) in the domain $\wido$ below the curve
$\wids$ described above. It is easy to see that, if
$\alpha_\pm:=\arctan q_\pm$, then the only solutions of
(\ref{ha})-(\ref{mh}) are given, for all $(X,Y)\in\wido$, by
\[ \tips(X,Y):=\beta\, \Im\left[i\left(ie^{\ds
i(\alpha_+-\alpha_-)/2}Z\right)^{\ds{\pi}/({\pi-(\alpha_++\alpha_-)})}\right],
\]
where $Z=X+iY$ and $\beta\geq 0$. It is straightforward to check
that, apart from the cases when either $q_\pm=0$ or
$q_\pm=\frac{1}{\sqrt 3}$, none of the above functions $\tips$
satisfies (\ref{aqw}). When $q_\pm=0$, the  only solution of
(\ref{aqw}) of the above type is $\tips_0:\equiv 0$ in $\wido$. When
$q_\pm=\frac{1}{\sqrt{3}}$, the only solution of (\ref{aqw}) of the
above type is the function $\tips^*$ given by (\ref{fisa}).

If $q_+\neq\infty$ and $q_-=\infty$ then, for the solution $(\wids,
\tips)$ of (\ref{blow}) given by Theorem \ref{tblow}, $\wids$
necessarily consists of the negative imaginary axis and the
half-line $\{(X,-q_+X):X\geq 0\}$. Arguing as before, a
contradiction is reached. A similar argument shows that it is also
not possible that $q_+=\infty$ and $q_-\neq\infty$.

The possibility that $q_\pm=\infty$ is ruled out by the argument
used to show that $\sigma=0$ in the proof of Theorem \ref{tblow}.

We conclude that necessarily either $q_\pm=\frac{1}{\sqrt{3}}$ or
$q_\pm=0$. When $\gamma(r)\geq 0$ for all $r\in[0,\delta]$, the
possibility that $q_\pm=0$ is ruled out by Proposition \ref{sdf}.
This completes the proof of Theorem \ref{asco}.
\end{proof}

\section{Appendix}

We recall the definition of a non-tangential limit and some notions
and results concerning the classical Hardy spaces of harmonic
functions. More details can be found in \cite{Du,Ko, Ru}. In what
follows, $\mcd$ denotes the unit disc in the plane and
$\mcd_\pm:=\mcd\cap\bdr^2_\pm$.

Let $\mcg$ be an open set in the plane. Let $(X_0, Y_0)\in\partial
\mcg$ be such that there exist an open set $\mcu$ containing $(X_0,
Y_0)$ and a homeomorphism $h:\mcd\to \mcu$  such that $h(\mcd_+)=
\mcg\cap \mcu$, $h((-1,1)\times\{0\})=\partial \mcg\cap \mcu$
 and the curve $\partial \mcg\cap \mcu$
has a tangent at $(X_0,Y_0)$. Let $\mathbf{n}$ be the unit inner
normal to $\mcg$ at $(X_0,Y_0)$. We say that a sequence $\{(X_n,
Y_n)\}\indn$ of points in $\mcg$ \emph{tends to $(X_0, Y_0)$
non-tangentially} if $(X_n, Y_n)\to (X_0,Y_0)$ as $n\to\infty$ and
there exists $\kappa>0$ such that
\[(X_n-X_0, Y_n-Y_0)\cdot\mathbf{n}\geq \kappa[(X_n-X_0)^2+(Y_n-Y_0)^2]^{1/2}
\quad\text{for all }n\geq 1,\] where $\cdot$ denotes the usual inner
product in $\bdr^2$.
 Let $f:\mcg\to\bdc$
and $l\in\bdc$. We say that $f$ \emph{has non-tangential limit $l$
at $(X_0,Y_0)$} if $\lim_{n\to\infty}f(X_n, Y_n)=l$ for every
sequence $\{(X_n, Y_n)\}\indn$  which tends to $(X_0, Y_0)$
non-tangentially.

 For $p\in[1,\infty)$, the
Hardy space $h^p_\bdc(\mcd)$ is usually defined as the class of
harmonic functions $f:\mcd\to\bdc$  with the property that
\be\sup_{r\in
(0,1)}\int_{-\pi}^\pi|f(re^{it})|^p\,dt<+\infty.\label{hagr}\ee The
Hardy space $h^\infty_\bdc(\mcd)$ is the class of bounded harmonic
functions in $\mcd$. For $p\in[1,\infty]$, the Hardy space
$\ha^p(\mcd)$ is the class of holomorphic functions in $\h^p(\mcd)$.
Fatou's Theorem states that any function in $h_\bdc^p(\mcd)$, $p\in
[1,\infty]$, has non-tangential limits almost everywhere on the unit
circle. The boundary values of any function in $H^1_\bdc(\mcd)$
cannot vanish on a set of positive measure unless the function is
identically $0$ in $\mcd$. The M. Riesz Theorem \cite[Theorem
4.1]{Du} states that, if $u\in \h^p(\mcd)$ for some $p\in
(1,\infty)$, and if $v$ is a harmonic function such that $u+iv$ is
holomorphic, then $v\in \h^p(\mcd)$. For any function
$f:\mcd\to\bdc$, the \emph{radial maximal function}
$M_{\textnormal{rad}}[f]$ is defined \cite[Definition 11.19]{Ru} by
\[M_{\textnormal{rad}}[f]:=\sup_{r\in[0,1)}|f(re^{it})|\quad\text{for all }t\in\bdr.\]
 If $f\in \ha^p(\mcd)$, where $p\in[1,\infty)$ then \cite[Theorem
 7.11]{Ru} shows that
$M_{\textnormal{rad}}[f]\in L^p_{2\pi}$, the space of
$2\pi$-periodic functions in $L^p_{\textnormal{\loc}}(\bdr)$.

 The definition of Hardy spaces in general
domains \cite[Ch.\ 10]{Du} is based on the fact that, for $p\in [1,
\infty)$, a harmonic function $f$ belongs to $h_\bdc^p(\mcd)$ if and
only if the subharmonic function $|f|^p$ has a harmonic majorant,
i.e. there exists a positive harmonic function $w$ in $\mcd$ such
that $|f|^p\leq w$ in $\mcd$.  Let $\mcg$ be an open set. For $p\in
[1,\infty)$, the space $\h^p(\mcg)$ is  the class of harmonic
functions $f:\mcg\to\bdc$ for which the subharmonic function $|f|^p$
has a harmonic majorant in $\mcg$. The Hardy space $\h^\infty(\mcg)$
is the class of bounded harmonic functions in $\mcg$. The spaces
$\ha^p(\mcg)$ consists of the holomorphic functions in $\h^p(\mcg)$,
for $p\in[1,\infty]$. It is easy to check that the Hardy spaces are
conformally invariant: if $\mcg_1$ and $\mcg_2$ are open sets, and
$\sigma:\mcg_1\to \mcg_2$ is a conformal mapping, then
$f\in\h^p(\mcg_2)$ if and only if $f\circ\sigma\in \h^p(\mcg_1)$,
where $p\in [1,\infty]$. For this reason, many properties of the
Hardy spaces of the disc extend by conformal mapping to Hardy spaces
of simply connected domains.
 If $\mcg$ is a bounded open set whose
boundary is a rectifiable Jordan curve, then any function in
$\h^p(\mcg)$, where $1\leq p\leq\infty$, has non-tangential boundary
values $\mch^1$-almost everywhere. A consequence of this is the
existence of non-tangential boundary values $\mch^1$-almost
everywhere for functions in $\h^p(\mcg)$, $1\leq p\leq\infty$, for
any open set $\mcg$ with the following property: for any $(X_0,
Y_0)\in\partial \mcg$ there exist an open set $\mcu$ containing
$(X_0, Y_0)$ and a homeomorphism $h:\mcd\to \mcu$ such that
$h(\mcd_+)=\mcg\cap \mcu$, $h((-1,1)\times\{0\})=\partial \mcg\cap
\mcu$ and the curve $\partial \mcg\cap \mcu$ is rectifiable.  If
$\mcg$ is a bounded open set whose boundary is a rectifiable Jordan
curve, then the non-tangential boundary values of any function in
$\ha^1(\mcg)$ cannot vanish on a set of positive $\mch^1$ measure
unless the function is identically $0$ in $\mcg$.

\end{document}